\documentclass[10pt,a4paper]{siamltex}
\usepackage{amsmath,amsfonts,amssymb,latexsym,epic,eepic}
\usepackage[dvips]{graphicx,graphics,epsfig,color}
\usepackage{ifthen}
\usepackage{tikz,pstricks,fp}
\newtheorem{remark}{Remark}[section]
\newcommand{\bfa}{{\boldsymbol a}}
\newcommand{\bfb}{{\boldsymbol b}}
\newcommand{\bfI}{{\boldsymbol I}}
\newcommand{\bfn}{{\boldsymbol n}}
\newcommand{\bfu}{{\boldsymbol u}}
\newcommand{\bfv}{{\boldsymbol v}}
\newcommand{\bfw}{{\boldsymbol w}}
\newcommand{\bfx}{{\boldsymbol x}}
\newcommand{\bfy}{{\boldsymbol y}}
\newcommand{\bftau}{{\boldsymbol \tau}}
\newcommand{\bfvphi}{{\boldsymbol \varphi}}
\newcommand{\dv}{\partial}
\newcommand{\lapi}{{\boldsymbol \Delta}}
\newcommand{\gradi}{{\boldsymbol \nabla}}
\newcommand{\dive}{{\rm div}}
\newcommand{\divv}{\boldsymbol{\rm div}}
\newcommand{\xN}{\mathbb{N}}
\newcommand{\xR}{\mathbb{R}}
\newcommand{\xL}{{\rm L}}
\newcommand{\xH}{{\rm H}}
\newcommand{\xbfH}{{\rm \bf H}}
\newcommand{\norm}[1]{|\hspace{-.1em}| #1 |\hspace{-.1em}|}
\newcommand{\brok}{{1,\disc}}
\newcommand{\edge}{\sigma}
\newcommand{\edges}{\mathcal{E}}
\newcommand{\edgesint}{{\mathcal E}_{{\rm int}}}
\newcommand{\edgesext}{{\mathcal E}_{{\rm ext}}}
\newcommand{\mesh}{{\mathcal M}}
\newcommand{\disc}{\mathcal{T}}
\newcommand{\edgeedgeprime}{D_\edge|D_{\edge'}}
\newcommand{\edged}{\varepsilon}
\newcommand{\edgesd}{\bar {\mathcal E}}
\newcommand{\edgesdint}{\bar{\mathcal E}_{{\rm int}}}
\newcommand{\edgesdext}{\bar{\mathcal E}_{{\rm ext}}}
\newcommand{\fluxd}{F_{\edge,\edged}}
\newcommand{\Ds}{{\scalebox{0.6}{$D_\edge$}}}
\newcommand{\eos}{\wp}
\newcommand{\ie}{\emph{i.e.}}
\newcommand{\eg}{\emph{e.g.}}
\newcommand{\exm}{^{(m)}}
\newcommand{\ma}{\varepsilon}
\newcommand{\mcal}[1]{\mathcal{#1}}

\newcommand{\deltap}{{\delta \hspace{-0.1em} p}}
\newcommand{\Ind}{\mathcal{X}}
\newcommand{\Ent}[1]{{\lfloor #1 \rfloor}}
\title{Low Mach number limit of some staggered schemes for compressible barotropic flows}
\author{
R. Herbin \thanks{I2M UMR 7373, Aix-Marseille Universit\'e, CNRS, \'Ecole Centrale de Marseille.
39 rue Joliot Curie. 13453, Marseille, France. ({\tt raphaele.herbin@univ-amu.fr})}
\and
J.-C. Latch\'e \thanks{Institut de Radioprotection et de S\^{u}ret\'{e} Nucl\'{e}aire (IRSN) ({\tt jean-claude.latche@irsn.fr})}
\and
K. Saleh \thanks{Universit\'e de Lyon, CNRS UMR 5208, Universit\'e Lyon 1, Institut Camille Jordan.
43 bd 11 novembre 1918; F-69622 Villeurbanne cedex, France. ({\tt saleh@math.univ-lyon1.fr})}
}

\begin{document}
\maketitle
\begin{center} \today \end{center}

\begin{abstract}
In this paper, we study the behaviour of some staggered discretization based numerical schemes for the barotropic Navier-Stokes equations at low Mach number.
Three time discretizations are considered: the implicit-in-time scheme and two non-iterative pressure correction schemes.
The two latter schemes differ by the discretization of the convection term: linearly implicit for the first one, so that the resulting scheme is unconditionally stable, and explicit for the second one, so the scheme is stable under a CFL condition involving the material velocity only.
We rigorously prove that these three variants are asymptotic preserving in the following sense: for a given mesh and a given time step, a sequence of solutions obtained with a sequence of vanishing Mach numbers tend to a solution of a standard scheme for incompressible flows.
This convergence result is obtained by mimicking the proof of convergence of the solutions of the (continuous) barotropic Navier-Stokes equations to that of the incompressible Navier-Stokes equation as the Mach number vanishes.
Numerical results performed with a hand-built analytical solution show the behaviour that is expected from the analysis.
\end{abstract}

\begin{keywords}
Compressible Navier-Stokes equations, low Mach number flows, finite volumes, Crouzeix-Raviart scheme, Rannacher-Turek scheme, finite elements, staggered discretizations.
\end{keywords}

\begin{AMS} 35Q31,65N12,76M10,76M12 \end{AMS}
% 65 : Numerical Analysis
% 65M : Partial differential equation, initial value and time-dependent initial-boundary value problems
% 65M12 : Stability and convergence of numerical methods
%
%-----------------------------------------------------------------------------------------------------------------------
%
\section{Introduction} \label{sec:intro}

We consider the non-dimensionalized system of time-dependent barotropic compressible Navier-Stokes equations, parametrized by the Mach number denoted thereafter by $\ma$, and posed for $(\bfx,t)\in\Omega\times (0,T)$: 
\begin{subequations}
\begin{align}\label{eq:pb_mass} &
\partial_t \rho^\ma + \dive( \rho^\ma\, \bfu^\ma) = 0,
\\[1ex] \label{eq:pb_mom} &
\partial_t (\rho^\ma\, \bfu^\ma) + \divv(\rho^\ma\, \bfu^\ma \otimes \bfu^\ma) 
 - \divv(\bftau(\bfu^\ma)) + \frac 1 {\ma^2}\ \gradi \eos(\rho^\ma) = 0,
\end{align} \label{eq:pb}\end{subequations}
where $T$ is a finite positive real number, and $\Omega$ is an open bounded connected subset of $\xR^d$, with $d\in\lbrace2,3\rbrace$, which is polygonal if $d=2$ and polyhedral if $d=3$.
The quantities $\rho^\ma>0$ and $\bfu^\ma=(u_1^\ma,..,u_d^\ma)^T$ are the density and velocity of the fluid.

In the isentropic case, the pressure satisfies the ideal gas law $\eos(\rho^\ma) = (\rho^\ma)^\gamma$, with $\gamma \geq 1$, the heat capacity ratio, a coefficient which is specific to the considered fluid.
However, more general barotropic cases can be considered provided the equation of state $\eos$  is a $C^1$ increasing convex function  such that $\eos'(1) >0$, see Remark \ref{rem-baro}.

Equation \eqref{eq:pb_mass} expresses the local conservation of the mass of the fluid while equation \eqref{eq:pb_mom} expresses the local balance between momentum and forces. 
We consider Newtonian fluids so that the shear stress tensor $\bftau(\bfu^\ma)$ satisfies:
\[
\divv(\bftau(\bfu)) = \mu \Delta \bfu+ (\mu+\lambda)\gradi (\dive \,\bfu),
\]
where $\mu$ and $\lambda$ are two parameters with $\mu>0$ and $\mu+\lambda>0$.
System \eqref{eq:pb} is complemented with the following boundary and initial conditions:
\begin{equation} \label{eq:pb_CI}
\rho^\ma|_{t=0} = \rho_0^\ma, \qquad \qquad \bfu^\ma|_{t=0}=\bfu_0^\ma, \qquad\qquad \bfu^\ma|_{\dv\Omega}=0.
\end{equation}
At the continuous level, when $\ma$ tends to zero, the density $\rho^\ma$ tends to a constant and the velocity tends, in a sense to be defined, to a solution of the incompressible Navier-Stokes equations \cite{lio-98-inc}.
Heuristically, the momentum equation \eqref{eq:pb_mom} indicates that $\rho^\ma$ behaves like $\bar \rho(t) + \mcal{O}(\ma^{\frac 2\gamma})$ where $\bar \rho(t)$ is a function only dependent on the time variable.
Integrating the mass conservation equation \eqref{eq:pb_mass} over $\Omega$ and using the homogeneous Dirichlet boundary condition \eqref{eq:pb_CI} (a homogeneous Neumann condition would be sufficient) then implies that $\bar \rho$ is actually a constant.
For such a result to hold, some assumptions need to be made on the initial data; in particular, the initial density $\rho_0^\ma$ must be assumed to be close to $\bar \rho$ in a certain sense.
These assumptions will be specified below.
Setting $\bar \rho=1$ (here and throughout the paper) without loss of generality, passing to the limit in the mass conservation equation \eqref{eq:pb_mass} and in the momentum balance \eqref{eq:pb_mom}, the limit velocity $\bar \bfu$ is formally seen to solve the system of incompressible Navier-Stokes equations:
\[
 \begin{aligned}
  &\dive( \bar \bfu) = 0, \\[1ex] 
  &\dv_t \bar \bfu + \divv(\bar \bfu \otimes \bar \bfu)- \mu \lapi \bar \bfu + \gradi \pi = 0,
 \end{aligned}
\]
where $\pi$ is the formal limit of $(\eos(\rho^\ma)-1)/\ma^2$.
This formal computation was justified by rigorous studies \cite{lio-98-inc, des-99-gre,des-99-mas}; see also \cite{ebi-77-mot,ebi-82-mot,kla-81-sing,kre-80-pro,sch-87-hyp} for some of the first mathematical analyses on low Mach number limits and  \cite{met-01-inc,dan-02-zer,gal-03-res,alaz-06-low,fei-07-asy,fei-07-low,fei-10-low} for some of the numerous related works.

\medskip
For the low Mach limit of numerical schemes for \eqref{eq:pb}, the issue is not so clear.
Indeed, in general, schemes designed for computing compressible flows do not boil down, when $\ma \rightarrow 0$, to standard schemes for incompressible flows, for essentially two reasons.
First, the numerical dissipation introduced to stabilize the scheme depends on the celerity of acoustic waves, which blows up when $\ma \rightarrow 0$; a reasonable approximation of the incompressible solution thus may need a very small space step, depending not only on the regularity of the continuous solution but also on $\ma$.
Second, these schemes are usually explicit in time (with a sophisticated derivation of fluxes, for instance through solutions of Riemann problems at interfaces, which lead to a nonlinear expression with respect to the unknowns), which is not compatible with a wave celerity blowing up in the incompressible limit; to cope with the low Mach number situation, an implicitation of some terms in the equations, usually the pressure gradient in the momentum balance equation and the mass fluxes divergence in the mass balance, is thus necessary, and this makes implicit-in-time discrete analogues of the wave equation for the pressure appear.
Unfortunately, the schemes for the compressible case usually use a collocated arrangement of the unknowns (specially if one intends to compute the numerical fluxes on the basis of the solution of a Riemann problem at faces, which is well suited to a cell by cell piecewise constant approximation), and the diffusion operator appearing in this wave equation (obtained by a discrete composition of the pressure gradient and the velocity divergence operators) is unstable, since collocated approximations do not satisfy a form of the so-called discrete {\em inf-sup} condition.
At the incompressible limit, the scheme will thus need an additional stabilization mechanism \cite{mog-12-pre}. 
These phenomena have been widely studied and corrections have been proposed \cite{tur-99-pre, gui-99-beh, gui-04-beh, gui-06-rec, del-10-ana, deg-11-all, cor-12-asy, haa-12-all, noe-14-wea,cha-16-all, zak-17-mac}.
To obtain a scheme accurate for all Mach number flows, an alternative route consists in starting from technologies initially designed for  the incompressible Navier-Stokes equations and extending them to compressible flows.
This approach may be traced back to the late sixties, when first attempts were done to build "all flow velocity" schemes \cite{har-68-num, har-71-num}; these algorithms may be seen as an extension to the compressible case of the celebrated MAC scheme, introduced some years before \cite{har-65-num, ara-81-pot}.
These seminal papers have been the starting point for the development of numerous schemes falling in the class of pressure correction algorithms (see \eg\ \cite{cho-68-num, tem-69-sur, gue-06-ove} for a presentation in the incompressible case), possibly iterative, in the spirit of the SIMPLE method, some of them based on staggered finite volume space discretizations \cite{cas-84-pre, iss-85-sol, iss-86-com, van-87-seg, kar-89-pre, mcg-90-sho, bij-98-uni, yoo-99-uni, col-99-pro, van-01-sta, wen-02-mac, wal-02-sem, van-03-con, vid-06-sup, kwa-09-met}; a bibliography extended to the schemes using other space discretizations may be found in \cite{her-14-ons}.

\medskip
In this paper, we address, besides a purely implicit scheme, variants of schemes falling in this latter class, namely non-iterative pressure correction schemes based on staggered discretizations.
These schemes have been developed in the last ten years, first for barotropic Euler and Navier-Stokes equations \cite{her-13-pre, her-14-ons} and then for the non-barotropic case \cite{her-14-ons, gra-16-unc}, and have been shown both theoretically and numerically to be consistent and accurate for the Navier-Stokes and Euler equations and for Mach numbers in the range of unity (including shock solutions in the inviscid cases).
In addition, some numerical experiments \cite{gra-16-unc} suggest that, when the Mach number tends to zero, the numerical solution tends to the solution of a standard scheme for incompressible or, in non-isothermal situations, quasi-incompressible flows (in the sense of the classical asymptotic model for low Mach numbers \cite{maj-85-the}).
Our aim here is to rigorously prove that, in the barotropic case, these schemes are indeed asymptotic preserving: for a given discretization (\ie\ mesh and time step), when the Mach number tends to zero, the solution tends to the solution of a standard (stable and accurate) scheme for incompressible flows.
To this purpose, we reproduce at the discrete level the analysis performed in the continuous case in \cite{lio-98-inc} (of course, with heavy simplifications on compactness arguments, especially concerning the compactness of the sequence of discrete velocities, since we work in a finite dimensional setting); to our knowledge, this is the first presentation of such a proof.
We draw the reader's attention of the fact that the results are mainly presented in the case of the Navier-Stokes equations, but we also explain how they can be easily extended to the inviscid case of the Euler equations where $\mu=\lambda=0$.  

\medskip
We address three different time-discretizations: first, a fully implicit scheme, because the convergence proof in this case is simpler and necessitates less restrictive assumptions on the initial data; then we turn to two variants which are more efficient in practice, namely two pressure correction schemes, which differ by the discretization of the convection term in the momentum balance equation, linearly implicit for the first one (so that the corresponding scheme is unconditionally stable) and fully explicit for the second one (so that the corresponding scheme is stable under a CFL condition based on the material velocity).
The paper is organized as follows.
First, for the reader's convenience, we recall the (part of) the continuous analysis which is mimicked at the discrete level (Section \ref{sec:cont}).
Then we define the meshes and unknowns used by the schemes (Section \ref{sec:mesh}) and the space discretization of the operators involved in the balance equations (Section \ref{sec:disc}).
The next three sections are devoted to the convergence analysis for the time-discretization variants: the implicit scheme (Section \ref{sec:impl}) and the two pressure correction schemes (Sections \ref{sec:proj1} and \ref{sec:proj2}).
The last section presents some numerical results for the pressure correction algorithm at various Mach numbers on a problem which is built upon a hand-built analytical solution.

%
% --------------------------------------------------------------------------------
%
\section{Incompressible limit in the continuous setting}\label{sec:cont}

The convergence when $\ma$ tends to zero of a weak solution $(\rho^\ma,\bfu^\ma)$ to the initial value problem \eqref{eq:pb}-\eqref{eq:pb_CI}  is the purpose of various papers published in the late $90'$ \cite{lio-98-inc, des-99-gre,des-99-mas}.
In this section, we wish to recall some of the key arguments that are used in these works to pass to the limit on the global weak solutions of \eqref{eq:pb}-\eqref{eq:pb_CI} as the Mach number vanishes.

\medskip
The results proven in the above-mentioned papers on the convergence of $(\rho^\ma,\bfu^\ma)$ towards a weak solution of the incompressible Navier-Stokes equations follow a two-step argument. 
The first step consists in deriving \emph{a priori} bounds on the quantities $\rho^\ma-1$ and $\bfu^\ma$ which are uniform with respect to $\ma$. These bounds imply the strong convergence of $\rho^\ma$ towards $\bar\rho=1$ in $\xL^\infty((0,T);\xL^\gamma(\Omega))$, and up to the extraction of a subsequence, the weak convergence in $\xL^2((0,T); \xH^1_0(\Omega))$ of $\bfu^\ma$ towards some function $\bar \bfu$.  
The second step consists in passing to the limit in the weak formulation of problem \eqref{eq:pb}-\eqref{eq:pb_CI} thanks to these convergence properties. 
The main difficulty in this step is the passage to the limit for the term $\divv(\rho^\ma \bfu^\ma\otimes\bfu^\ma)$ with only a weak convergence of the velocity. 

\medskip
Subsequently, we describe the main arguments to obtain the estimates on $\rho^\ma-1$ and $\bfu^\ma$, which we will mimick at the discrete level in order to prove the asymptotic preserving feature of the staggered schemes in the low Mach number limit.
However, we do not need a careful study of the nonlinear term since at the discrete level, \emph{i.e.} for a fixed mesh, all norms are equivalent and the boundedness of $(\bfu^\ma)_{\ma>0}$ is enough to obtain convergence in any finite dimensional norm up to the extraction of a subsequence, and then to pass to the limit on the numerical scheme.
%
% ---------------
%
\subsection{\emph{A priori} estimates}

We begin by recalling some key identities satisfied by the smooth solutions of \eqref{eq:pb}, which are then incorporated in the definition of weak solutions.

\begin{proposition}
Let $\psi_\gamma$ be the function defined for $\rho>0$ as $\psi_\gamma(\rho)=\rho\log\rho$ if $\gamma=1$, and $\psi_\gamma(\rho)=\rho^\gamma/(\gamma-1)$ if $\gamma>1$ and define $\Pi_\gamma(\rho)=\psi_\gamma(\rho)-\psi_\gamma(1)-\psi_\gamma'(1)(\rho-1)$. The smooth solutions of \eqref{eq:pb}-\eqref{eq:pb_CI} satisfy the following identities:
\begin{itemize}
 \item \emph{A kinetic energy balance}:
  \begin{equation}
  \label{eq:pb_ke}
   \dv_t (\frac 1 2 \rho^\ma\, |\bfu^\ma|^2) + \dive(\frac 1 2 \rho^\ma\, |\bfu^\ma|^2\ \bfu^\ma) 
  - \divv(\bftau(\bfu^\ma)) \cdot \bfu^\ma
  + \frac 1 {\ma^2}\ \gradi \eos(\rho^\ma)\cdot \bfu^\ma=0. 
  \end{equation}
 \item \emph{A renormalization identity}
 \begin{equation}
  \label{eq:pb_renorm}
  \dv_t \psi_\gamma(\rho^\ma) + \dive \bigl(\psi_\gamma(\rho^\ma)\, \bfu^\ma \bigr) + \eos(\rho^\ma)\, \dive  \bfu^\ma = 0. 
  \end{equation}
 \item \emph{A "positive" renormalization identity}:
 \begin{equation}
  \label{eq:pb_renorm:Pi}
  \dv_t \Pi_\gamma(\rho^\ma) + \dive \bigl(\psi_\gamma(\rho^\ma)\, \bfu^\ma - \psi_\gamma'(1)\rho^\ma \bfu^\ma\bigr) + \eos(\rho^\ma)\, \dive  \bfu^\ma = 0. 
  \end{equation}
  \item \emph{An entropy identity}:
    \begin{multline}
    \label{entrop_2}
    \dv_t (\frac 1 2 \rho^\ma\, |\bfu^\ma|^2)  
    + \frac 1 {\ma^2}\ \dv_t \Pi_\gamma(\rho^\ma) 
    +\dive \Big( \big (\frac 1 2 \rho^\ma\, |\bfu^\ma|^2 
    + \frac 1 {\ma^2}\ \psi_\gamma(\rho^\ma) 
    \\- \frac 1 {\ma^2} \psi_\gamma'(1)\rho^\ma 
    + \frac 1 {\ma^2} \eos(\rho^\ma) \big ) \bfu^\ma \Big ) 
    - \divv(\bftau(\bfu^\ma)) \cdot \bfu^\ma = 0.
    \end{multline}
\end{itemize}
\end{proposition}

\begin{proof}
The proof of \eqref{eq:pb_ke} and \eqref{eq:pb_renorm} are classical. Multiplying the mass conservation equation \eqref{eq:pb_mass} by $-\psi_\gamma'(1)$ and summing with \eqref{eq:pb_renorm} yields \eqref{eq:pb_renorm:Pi}. 
Summing \eqref{eq:pb_ke} and $\ma^{-2}\times$\eqref{eq:pb_renorm:Pi} yields \eqref{entrop_2}.
\end{proof}

\begin{remark}
We call \eqref{eq:pb_renorm:Pi} a "positive" renormalization identity because the function $\Pi_\gamma$ is positive thanks to the convexity of $\psi_\gamma$.
\end{remark}

Integrating \eqref{entrop_2} over $\Omega\times(0,t)$ and recalling the homogeneous Dirichlet boundary conditions on the velocity, we obtain the following estimate on the solution pair $(\rho^\ma,\bfu^\ma)$.
For all $t\in (0,T)$:
\begin{multline}
\label{estimate_Pi_eq}
\frac 12 \int_\Omega \rho^\ma(t)\, |\bfu^\ma(t)|^2 +  \frac 1 {\ma^2} \int_\Omega \Pi_\gamma(\rho^\ma(t)) + \mu \int_0^t\norm{\gradi \bfu^\ma(s)}^2_{\xL^2(\Omega)^{d\times d}}{\rm d}s \\ + (\mu+\lambda)  \int_0^t  \norm {\dive (\bfu^\ma(s))}_{\xL^2(\Omega)}^2 {\rm d}s
 \ = \  \frac 12 \int_\Omega \rho_0^\ma\, |\bfu_0^\ma|^2 +  \frac 1 {\ma^2} \int_\Omega \Pi_\gamma(\rho_0^\ma).
\end{multline}
%
% ---------------
%
\subsection{Asymptotic behavior of the density and velocity in the zero Mach limit}

The global entropy estimate \eqref{estimate_Pi_eq} is thus proven for any strong solution of the boundary and initial value problem \eqref{eq:pb}-\eqref{eq:pb_CI}. In the following, we always \emph{assume} the existence of a weak solution $(\rho^\ma,\bfu^\ma)$ to problem \eqref{eq:pb}-\eqref{eq:pb_CI} satisfying the following inequality:
\begin{multline}
\label{estimate_Pi}
\frac 12 \int_\Omega \rho^\ma(t)\, |\bfu^\ma(t)|^2 +  \frac 1 {\ma^2} \int_\Omega \Pi_\gamma(\rho^\ma(t)) + \mu \int_0^t\norm{\gradi \bfu^\ma(s)}^2_{\xL^2(\Omega)^{d\times d}}{\rm d}s 
 \leq \  \frac 12 \int_\Omega \rho_0^\ma\, |\bfu_0^\ma|^2 +  \frac 1 {\ma^2} \int_\Omega \Pi_\gamma(\rho_0^\ma).
\end{multline}
In particular, we assume that all the integrals involved in \eqref{estimate_Pi} are convergent.

\medskip
Estimate \eqref{estimate_Pi} allows to prove the convergence of $\rho^\ma$ towards $1$ in $\xL^\infty((0,T);\xL^\gamma(\Omega))$ provided that the right hand side, which depends on the initial conditions is uniformly bounded with respect to $\ma$.
We show in the following that it is indeed the case provided some assumptions on the initial data.

\medskip
We begin by proving the following lemma, which states crucial properties of the function $\Pi_\gamma$.
\begin{lemma}[Estimates on $\Pi_\gamma$]
\label{lem:pi}
The function $\Pi_\gamma$ has the following lower bounds:
\begin{subequations}
 \begin{align}
    & \label{lowerPi0}
    \begin{array}{l}
      \hspace{-0.65ex} \bullet \ \mbox{For all $\gamma \geq 1$ and $\delta>0$, there exists $C_{\gamma,\delta}>0$ such that:} \\
      \hspace{2cm} \begin{array}{ll}
      \Pi_\gamma(\rho) \geq C_{\gamma,\delta}\ |\rho-1|^\gamma,     & \ \forall\rho>0 \ \text{with} \ |\rho-1|\geq\delta, 
      \end{array} 
      \end{array} \\[1ex]
    & \label{lowerPi1} \bullet \  \mbox{If $\gamma \geq 2$ then} \ \Pi_\gamma(\rho) \geq |\rho-1|^2, \ \forall \rho>0. \\[1ex]
    &  \label{lowerPi2} 
    \begin{array}{l}
      \hspace{-0.65ex} \bullet \ \mbox{If $\gamma \in [1,2)$ then for all $R\in(2,+\infty)$, there exists $C_{\gamma,R}$ such that:} \\
      \hspace{2cm} \begin{array}{ll}
      \Pi_\gamma(\rho) \geq C_{\gamma,R}\ |\rho-1|^2,     & \ \forall\rho\in (0,R), \\ 
      \Pi_\gamma(\rho) \geq C_{\gamma,R}\ |\rho-1|^\gamma, & \ \forall  \rho\in [R,\infty).
      \end{array} 
      \end{array}
\end{align}
\end{subequations}
Moreover, the function $\Pi_\gamma$ has the following upper bound (for small densities): For all $\gamma \geq 1$ there exists $C_\gamma$ such that:
\begin{equation}
 \label{upperPi} \Pi_\gamma(\rho) \leq C_\gamma \,|\rho-1|^2,\ \forall\rho\in (0,2).
\end{equation}
\end{lemma}

\begin{proof}
For $\gamma=1$, we have $\Pi_1=\rho\log\rho-\rho$. Hence $\Pi_1\sim\rho\log\rho$ for large values of $\rho$, which implies \eqref{lowerPi0}. Similarly, for $\gamma>1$, we have $\Pi_\gamma(\rho)=\psi_\gamma(\rho)-\psi_\gamma(1)-\psi_\gamma'(1)(\rho-1)=(\gamma-1)^{-1}(\rho^\gamma-1-\gamma(\rho-1))$, thus $\Pi_\gamma(\rho)\sim (\gamma-1)^{-1} \rho^\gamma$ for large values of $\rho$ which proves \eqref{lowerPi0}.
A second order Taylor expansion of $\psi_\gamma$ yields, for all $\gamma \geq 1$:
 \[
  \Pi_\gamma(\rho) =  |\rho-1|^2 \, \gamma \int_0^1 (1+s(\rho-1))^{\gamma-2}(1-s){\rm d}s, \qquad \text{for all} \quad \rho\in(0,+\infty).
 \]
The case $\gamma \geq 2$ is straightforward and we obtain $|\rho-1|^2\leq \Pi_\gamma(\rho)$ for all $\rho\in(0,+\infty)$ and $\Pi_\gamma(\rho)\leq C_\gamma|\rho-1|^2$ for all $\rho \leq 2$ with $C_\gamma = \gamma \int_0^1 (1+s)^{\gamma-2}(1-s){\rm d}s$. For $1\leq \gamma < 2$, we easily get  $\Pi_\gamma(\rho)\leq |\rho-1|^2$ for all $\rho\in(0,+\infty)$ and the lower bound is obtained by separating the case $\rho < R$ and $\rho \geq R$. We obtain the expected lower bound \eqref{lowerPi2} with 
\[
C_{\gamma,R}=\gamma \int_0^1  \frac{1-s}{(1+s(R-1))^{2-\gamma}}{\rm d}s.
\]
\end{proof}
\begin{remark}[The barotropic case]
\label{rem-baro}
	Lemma \ref{lem:pi} \ considers the isentropic case $\eos (\rho) = \rho^\gamma$. 
	However, the results of sections \ref{sec:impl}-\ref{sec:proj2} hold in a more general barotropic case.
	Indeed, let $\eos$ be a continuous and derivable function. 
	Defining the functions $\psi(\rho)=\rho \int \frac{\eos(s)}{s^2} \ ds$ and $\Pi(\rho)=\psi(\rho)-\psi(1)-\psi'(1)(\rho-1)$, easy computations show that the renormalization identities \eqref{eq:pb_renorm} and \eqref{eq:pb_renorm:Pi} are still valid. Moreover,
% 	$\Pi$ as 
% 		\[\Pi(\rho)=\psi(\rho)-\psi(1)-\psi'(1)(\rho-1) \mbox{ with } \psi(\rho) = \rho \int \frac{\eos(s)}{s^2} \ ds,
% 		\]
	a straightforward calculation shows that 
	\[
		\Pi(\rho) = (\rho-1)^2 \int_0^1 \frac{\eos'(s(\rho-1)+1)}{s(\rho-1)+1} (1-s) \ ds.
	\]
	 Hence, if $\eos$ is an non-decreasing $C^1$ function such that $\eos'(1) >0$,  the above integral is positive and it is a continuous function of $\rho$; thus, under these assumptions on $\eos$, there exists $\underline C_{\eos,R}$ and $\overline C_{\eos,R} \in \xR_+$ such that 
	\[
		\underline C_{\eos,R} (\rho-1)^2  \le \Pi_\gamma(\rho) \leq \overline C_{\eos,R} (\rho-1)^2 \mbox{ for } | \rho | \le R. 
	\]
Note that in the discrete setting considered below, this estimate is sufficient since the discrete density $\rho$ is bounded uniformly with respect to the Mach number. 
Indeed, the conservative discretisation of the mass balance (see Section \ref{sec:mass}) yields that the discrete density $\rho$ satisfies $\int_\Omega \rho(x,t) \ dx = \int_\Omega \rho_0(x) \ dx$, so that on a given mesh, $\rho$ is bounded by $\frac 1 {|\underline K|}\int_\Omega \rho_0(x) \ dx$  where $|\underline K|$ is the measure of the smallest cell. 

The results of sections \ref{sec:impl}-\ref{sec:proj2} are thus still valid in the barotropic case, for a non-decreasing $C^1$ function $\eos$  such that $\eos'(1) >0$.
\end{remark}

Let us now assume that the initial data is "ill-prepared", in the following sense: $\rho_0^\ma \in \xL^\infty(\Omega)$ with  $\rho_0^\ma >0$ for a.e. $\bfx\in\Omega$, $\bfu_0^\ma \in \xL^2(\Omega)^d$ and there exists $C$ independent of $\ma$ such that: 
\begin{equation}
 \label{eq:u0_rho0_ip}
 \norm{\bfu_0^\ma}_{\xL^2(\Omega)^d} \ + \ \frac{1}{\ma}\, \norm{\rho^\ma_0- 1}_{\xL^\infty(\Omega)} \ \leq \ C.
\end{equation}
This bound implies that $\rho_0^\ma$ tends to $1$ in $\xL^\infty(\Omega)$ when $\ma\to 0$; moreover, we suppose that $\bfu_0^\ma$ weakly converges in $\xL^2(\Omega)^d$ towards a function $\bar \bfu_0\in\xL^2(\Omega)^d$. The initial data is said to be "ill-prepared" since $\rho^\ma_0-1$ behaves like $\ma$ and not like $\ma^{\frac 2\gamma}$ (or $\ma^2$) as suggested by the momentum equation, and the initial velocity is \emph{not required} to be close to a divergence free velocity.

\medskip
Under assumption \eqref{eq:u0_rho0_ip}, an easy consequence of the upper bound \eqref{upperPi} on $\Pi_\gamma$ is that the right-hand side of the estimate \eqref{estimate_Pi} is bounded independently of the Mach number $\ma$.
We then obtain that for any weak solution $(\rho^\ma,\bfu^\ma)$ to \eqref{eq:pb}-\eqref{eq:pb_CI} that satisfies the global entropy estimate \eqref{estimate_Pi}, the velocity $\bfu^\ma$ is bounded in $\xL^2((0,T);\xH_0^1(\Omega)^d)$ uniformly with respect to $\ma$.

\medskip
A further consequence of the lower bound \eqref{lowerPi0} on $\Pi_\gamma$ is the convergence of $\rho^\ma$ towards $1$ as $\ma\to0$ in $\xL^\infty((0,T);\xL^\gamma(\Omega))$ as stated in the following proposition.

\begin{proposition}
\label{conve_rho_continuous_0}
Let, $(\rho_0^\ma,\bfu_0^\ma)$ be a family of ill-prepared initial data and let $(\rho^\ma,\bfu^\ma)$ be a corresponding family of weak solutions of \eqref{eq:pb}-\eqref{eq:pb_CI} that satisfy the global entropy estimate \eqref{estimate_Pi}. Then, $\rho^\ma$ converges towards $1$ as $\ma\to0$ in $\xL^\infty((0,T);\xL^\gamma(\Omega))$.
\end{proposition}

\begin{proof}
By \eqref{lowerPi0} and estimate \eqref{estimate_Pi}, we have for all $\delta>0$, and $t>0$:
 \[
  \norm{\rho^\ma(t)-1}_{\xL^\gamma(\Omega)}^\gamma 
  \leq |\Omega|\, \delta^\gamma + \int_{\Omega} |\rho^\ma(t)-1|^\gamma \, \Ind_{\lbrace|\rho^\ma-1|\geq \delta \rbrace} 
  \leq |\Omega|\, \delta^\gamma + \frac{C\, \ma^2}{C_{\gamma,\delta}}.
 \]
where, for a given set $A$, $\Ind_{A}$ denotes the characteristic function of $A$.
Hence,
\[
\limsup\limits_{\ma\to 0} \norm{\rho^\ma-1}_{\xL^\infty((0,T);\xL^\gamma(\Omega))}\leq |\Omega|^{\frac 1 \gamma}\, \delta
\]
for all $\delta>0$, which concludes the proof.
\end{proof}

The following proposition provides a rate of convergence in $\xL^\infty((0,T);\xL^q(\Omega))$ of $\rho^\ma$ towards $1$ for $q\in[1,\min(2,\gamma)]$.

\begin{proposition} \label{conve_rho_continuous}
Let, $(\rho_0^\ma,\bfu_0^\ma)$ be a family of ill-prepared initial data and let $(\rho^\ma,\bfu^\ma)$ be a corresponding family of weak solutions of \eqref{eq:pb}-\eqref{eq:pb_CI} that satisfy the global entropy estimate \eqref{estimate_Pi}.
Then, the following estimates hold.
\begin{itemize}
\item[$\bullet$] If $\gamma \geq 2$, then there exists $C>0$ such that, for $\ma$ small enough:
\[
\norm{\rho^\ma-1}_{\xL^\infty((0,T);\xL^2(\Omega))} \leq C\ma.
\]
\item[$\bullet$] If $1\leq \gamma < 2$, then for $\ma$ small enough, for all $R\in(2,+\infty)$, there exists $C_R>0$ such that:
\[
\ma^{-1}\norm{(\rho^\ma-1)\Ind_{\lbrace \rho^\ma<R \rbrace}}_{\xL^\infty((0,T);\xL^2(\Omega))}
+  \ma^{-\frac 2 \gamma}\norm{(\rho^\ma-1)\Ind_{\lbrace \rho^\ma \geq R \rbrace}}_{\xL^\infty((0,T);\xL^\gamma(\Omega))}\leq C_R.  
\]
\end{itemize}

 As a consequence, for all $q\in[1,\min(2,\gamma)]$, there exists $C>0$ such that for $\ma$ small enough:
 \[
   \norm{\rho^\ma-1}_{\xL^\infty((0,T);\xL^q(\Omega))} \leq C\ma.
  \]
\end{proposition}

\begin{proof}
As already stated, thanks to \eqref{eq:u0_rho0_ip} and using the upper bound on $\Pi_\gamma(\rho)$ for small values of $\rho$, for $\ma$ small enough, the right hand side of \eqref{estimate_Pi} is bounded by some constant $C_0$, independent of $\ma$.
We now use the lower bounds on $\Pi_\gamma(\rho)$.
For $\gamma \geq 2$, by Lemma \ref{lem:pi} combined with estimate \eqref{estimate_Pi}, we have for all $t\in(0,T)$:
\[
 \norm{\rho^\ma(t)-1}_{\xL^2(\Omega)}^2 \leq  \int_\Omega \Pi_\gamma(\rho^\ma(t))\leq \,C_0\, \ma^2.
\]
For $1\leq\gamma < 2$, invoking once again Lemma \ref{lem:pi} and estimate \eqref{estimate_Pi}, we obtain for all $t\in(0,T)$ and for all $R\in(2,+\infty)$:
\[
\begin{array}{ll}
 (i) & \displaystyle \quad  \norm{(\rho^\ma(t)-1)\Ind_{\lbrace \rho^\ma(t)\leq R\rbrace}}_{\xL^2(\Omega)}^2 \leq \, \frac{1}{C_{\gamma,R}} \int_\Omega \Pi_\gamma(\rho^\ma(t))  \leq C \, \ma^2, \\[2ex]
 (ii) & \displaystyle \quad \norm{(\rho^\ma(t)-1)\Ind_{\lbrace \rho^\ma(t)\geq R \rbrace} }_{\xL^\gamma(\Omega)}^\gamma \leq \, \frac{1}{C_{\gamma,R}} \int_\Omega \Pi_\gamma(\rho^\ma(t))  \leq \, C \, \ma^2,
 \end{array}
\]
which concludes the proof.
\end{proof}
%
%-----------------------------------------------------------------------------------------------------------------------
%
\section{Meshes and unknowns}\label{sec:mesh} 

\begin{definition}[Staggered mesh] \label{def:disc}
A staggered discretization of $\Omega$, denoted by $\disc$, is given by a couple $\disc=(\mesh,\edges)$, where:
\begin{itemize}
\item $\mesh$, the primal mesh, is a finite family composed of non empty triangles and convex quadrilaterals for $d=2$ or non empty tetrahedra and convex hexahedra for $d=3$. The primal mesh $\mesh$ is assumed to form a partition of $\Omega$ : $\overline{\Omega}= \displaystyle{\cup_{K \in \mesh} \overline K}$.
For any $K\in\mesh$, let $\dv K  = \overline K\setminus K$ be the boundary of $K$, which is the union of cell faces.
We denote by $\edges$ the set of faces of the mesh, and we suppose that two neighbouring cells share a whole face: for all $\edge\in\edges$, either $\edge\subset \dv\Omega$ or there exists $(K,L)\in \mesh^2$ with $K \neq L$ such that $\overline K \cap \overline L  = \overline \edge$; we denote in the latter case $\edge = K|L$.
We denote by $\edgesext$ and $\edgesint$ the set of external and internal faces: $\edgesext=\lbrace \edge \in \edges, \edge \subset \dv \Omega \rbrace$ and $\edgesint=\edges \setminus \edgesext$.
For $K \in \mesh$, $\edges(K)$ stands for the set of faces of $K$. The unit vector normal to $\edge \in \edges(K)$ outward $K$ is denoted by $\bfn_{K,\edge}$.
In the following, the notation $|K|$ or $|\edge|$ stands indifferently for the $d$-dimensional or the $(d-1)$-dimensional measure of the subset $K$ of $\xR^d$ or $\edge$ of $\xR^{d-1}$ respectively.
\medskip
\item We define a dual mesh associated with the faces $\edge\in\edges$ as follows.
When $K\in\mesh$ is a simplex, a rectangle or a cuboid, for $\edge \in \edges(K)$, we define $D_{K,\edge}$ as the cone with basis $\edge$ and with vertex the mass center of $K$ (see Figure \ref{fig:mesh}).
We thus obtain a partition of $K$ in $m$ sub-volumes, where $m$ is the number of faces of $K$, each sub-volume having the same measure $| D_{K,\edge}|= |K|/m$.
We extend this definition to general quadrangles and hexahedra, by supposing that we have built a partition still of equal-volume sub-cells, and with the same connectivities.
The volume $D_{K,\edge}$ is referred to as the half-diamond cell associated with $K$ and $\edge$.
For $\edge \in \edgesint$, $\edge=K|L$, we now define the diamond cell $D_\edge$ associated with $\edge$ by $D_\edge=D_{K,\edge} \cup D_{L,\edge}$.
We denote by $\edgesd(D_\edge)$ the set of faces of $D_\edge$, and by $\edged=D_\edge|D_{\edge'}$ the face separating two diamond cells $D_\edge$ and $D_{\edge'}$.
As for the primal mesh, we denote by $\edgesdint$ the set of dual faces included in the domain and by $\edgesdext$ the set of dual faces lying on the boundary $\dv \Omega$.
In this latter case, there exists $\edge\in\edgesext$ such that $\edged=\edge$.
\end{itemize}
\end{definition}

\medskip
Relying on this definition, we now define a staggered space discretization.
The degrees of freedom for the density (\ie\ the discrete density unknowns) are associated with the cells of the mesh $\mesh$, and are denoted by:
\[
\big\{ \rho_K,\ K \in \mesh \big\},
\]
while the degrees of freedom for the velocity are located at the center of the faces of the mesh $\mesh$ and are therefore associated with the cells of the dual mesh $D_\edge$, $\edge\in\edges$ (as in the low-degree nonconforming finite-element discretizations proposed in \cite{cro-73-con, ran-92-sim}).
The Dirichlet boundary conditions are taken into account by setting the velocity unknowns associated with an external face to zero, so the set of discrete velocity unknowns reads:
\[
\lbrace \bfu_\edge \in \xR^d,\ \edge \in \edgesint\rbrace.
\]

We associate functions with the discrete unknowns of the schemes described hereinafter. To this purpose, we define the following sets of discrete functions of the space variable.

\begin{definition}[Discrete functional spaces]\label{def:disc_space}
Let $\disc=(\mesh,\edges)$ be a staggered discretization of $\Omega$ as defined in Definition \ref{def:disc}.
\begin{itemize}
 \item We denote by $\xL_\mesh(\Omega)\subset \xL^\infty(\Omega)$ the space of scalar functions which are piecewise constant on each primal mesh cell $K\in\mesh$.
For all $w\in \xL_\mesh(\Omega)$ and for all $K\in\mesh$, we denote by $w_K$ the constant value of $w$ in $K$, so the function $w$ reads:
\[
w(\bfx)= \sum_{K \in \mesh} w_K\, \mathcal{X}_K(\bfx) \qquad \mbox{for a.e. } \bfx \in \Omega,
\]
where $\mathcal{X}_K$ stands for the characteristic function of $K$.\\[0.5ex]
\item We denote by $\xH_\edges(\Omega)\subset \xL^\infty(\Omega)$ the space of scalar functions which are piecewise constant on each diamond cell of the dual mesh $D_\edge,~\edge\in\edges$.
For all $u \in \xH_\edges(\Omega)$ and for all $\edge\in\edges$, we denote by $u_\edge$ the constant value of $u$ in $D_\edge$, so the function $u$ reads:
\[
u(\bfx)= \sum_{\edge\in\edges} u_\edge\, \mathcal{X}_{D_\edge}(\bfx) \qquad \mbox{for a.e. } \bfx \in \Omega,
\]
where $\mathcal{X}_{D_\edge}(\bfx)$ stands for the characteristic function of $D_\edge$. We denote by $\xbfH_\edges(\Omega)=\xH_\edges(\Omega)^d$ the space of vector valued (in $\xR^d$) functions that are constant on each diamond cell $D_\edge$. Finally,
we denote $\xH_{\edges,0}(\Omega)=\bigl\lbrace u\in \xH_\edges(\Omega),\ u_\edge=0 \text{ for all } \edge \in \edgesext \bigr\rbrace$ and $\xbfH_{\edges,0}(\Omega)=\xH_{\edges,0}(\Omega)^d$ .
\end{itemize}
\end{definition}

\begin{figure}[tb]
\begin{center}
\includegraphics[width=10cm,height=6cm]{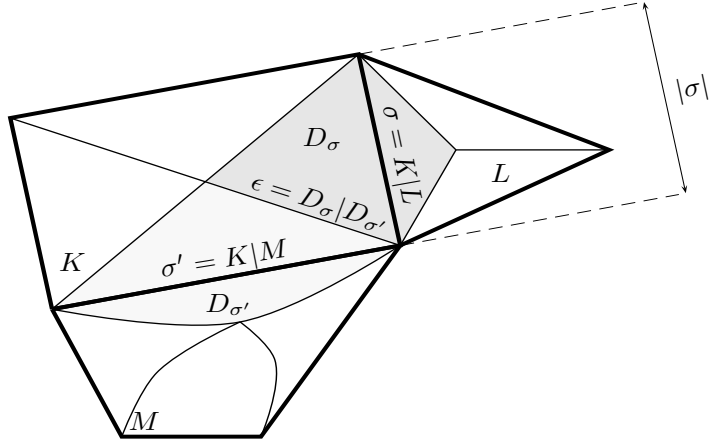}
\caption{Notations for control volumes and dual cells.}
\label{fig:mesh}
\end{center}
\end{figure}

%
%-----------------------------------------------------------------------------------------------------------------------
%
\section{Space discretization} \label{sec:disc} 

This section is devoted to the construction of the discrete space differential operators that approximate the differential operators in \eqref{eq:pb}.
As already said, the discretization is staggered.
The discrete operators involved in the discretization of the mass equation \eqref{eq:pb_mass} are thus associated with the cells of the primal mesh $K$, $K \in \mesh$, while the discrete operators involved in the discretization of the momentum equation \eqref{eq:pb_mom} are associated with the cells of the dual mesh $D_\edge$, $\edge\in\edgesint$.
%
% ---------------
%
\subsection{Mass convection flux} \label{sec:mass}
The discretization of the convection term $\dive(\rho\bfu)$ in the mass conservation equation is defined as follows. Given a discrete density field $\rho\in\xL_\mesh(\Omega)$ and a velocity field $\bfu\in\xbfH_{\edges,0}(\Omega)$, it is a piecewise constant function on each primal cell $K\in\mesh$ given by:
\begin{equation}
\label{div_mass}
\dive(\rho \bfu)_K = \frac 1 {|K|} \sum_{\edge \in\edges(K)} F_{K,\edge}(\rho,\bfu), \qquad \forall K \in \mesh.
\end{equation}
The quantity $F_{K,\edge}(\rho,\bfu)$ stands for the mass flux across $\edge$ outward $K$.
By the impermeability boundary conditions, it vanishes on external faces and is given on internal faces by:
\begin{equation}\label{eq:def_FKedge}
F_{K,\edge}(\rho,\bfu)= |\edge|\ \rho_\edge\ \bfu_\edge\cdot \bfn_{K,\edge}, \qquad \forall \edge\in\edgesint,\, \edge=K|L.
\end{equation}
The density at the face $\edge=K|L$ is approximated by the upwind technique, \ie\ $\rho_\edge=\rho_K$ if $\bfu_\edge\cdot \bfn_{K,\edge} \geq 0$ and $\rho_\edge=\rho_L$ otherwise.
%
% ---------------
%
\subsection{Velocity convection operator} \label{sec:vel_conv_op}

We now describe the approximation of the convection operator $\partial_t(\rho \bfu) + \divv(\rho \bfu \otimes \bfu)$ appearing in the momentum balance equation.
The approximation of the time derivative part $\partial_t(\rho \bfu)$ is naturally discretized at the dual cells $D_\edge$, $\edge\in\edgesint$ and an approximation $\rho_\Ds$ of the density on these dual cells $D_\edge$ is thus needed.
Given a density field $\rho \in\xL_\mesh(\Omega)$, this approximation is built as follows:
\begin{equation}\label{eq:def_rho}
|D_\edge|\, \rho_\Ds = |D_{K,\edge}|\, \rho_K + |D_{L,\edge}|\, \rho_L,  \qquad \forall \edge \in \edgesint,\ \edge=K|L.
\end{equation}

\medskip
In order to consider different time discretizations, we build a space discretization of a more general divergence part of the velocity convection operator, $\divv(\rho \bfu \otimes \bfv)$.
Given a discrete density field $\rho \in\xL_\mesh(\Omega)$, and two discrete velocity fields $\bfu\in\xbfH_{\edges,0}(\Omega)$ and $\bfv\in\xbfH_{\edges,0}(\Omega)$, this approximation is built as follows:
\begin{equation} \label{eq:div_conv}
\divv(\rho \bfu \otimes \bfv)_\edge= \frac 1 {|D_\edge|} \sum_{\edged\in\edgesd(D_\edge)} \fluxd(\rho,\bfu)\ \bfv_\edged, \qquad  \forall \edge \in \edgesint.
\end{equation}
$\fluxd(\rho,\bfu)$ is the mass flux across the edge $\edged$ of the dual cell $D_\edge$. Its value is zero if $\edged\in\edgesdext$. Otherwise, it is defined as a linear combination, with constant coefficients, of the primal mass fluxes at the neighboring faces. 
For $K \in \mesh$ and $\edge \in \edges(K)$, let $\xi_K^\edge$ be given by:
\[
\xi_K^\edge=\frac{|D_{K,\edge}|}{|K|},
\]
so that $\sum_{\edge \in \edges(K)} \xi_K^\edge=1$.
With the definition of the dual mesh adopted here, the value of the coefficients $\xi_K^\edge$ only depend on the type of the cell $K$ (simplicial or quandrangular/hexahedral).
For the quadrangular and hexahedral elements, we have $\xi_K^\edge=1/(2d)$ and, for the simplicial elements, $\xi_K^\edge=1/(d+1)$.
Then the mass fluxes through the inner dual faces are supposed to satisfy the following properties.
\begin{itemize}
\item[(H1)] The discrete mass balance over the half-diamond cells is satisfied, in the following sense.
For all primal cell $K$ in $\mesh$, the set $(\fluxd(\rho,\bfu))_{\edged\subset K}$ of dual fluxes included in $K$ solves the following linear system
\begin{equation}\label{eq:F_syst}
F_{K,\edge}(\rho,\bfu) + \sum_{\edged \in \edgesd(D_\sigma),\ \edged \subset K} F_{\edge,\edged}(\rho,\bfu)=
\xi_K^\edge \sum_{\edge' \in \edges(K)} F_{K,\edge'}(\rho,\bfu), \quad \edge \in \edges(K).
\end{equation}
\item[(H2)] The dual fluxes are conservative, \ie\ for any dual face $\edged=D_\edge|D_\edge'$, we have $F_{\edge,\edged}(\rho,\bfu)=-F_{\edge',\edged}(\rho,\bfu)$.
\item[(H3)] The dual fluxes are bounded with respect to the primal fluxes $(F_{K,\edge}(\rho,\bfu))_{\edge \in \edges(K)}$, in the sense that
\begin{equation}\label{eq:F_bounded}
|F_{\edge,\epsilon}(\rho,\bfu)| \leq \ \max \,\left \lbrace |F_{K,\edge'}(\rho,\bfu)|,\ \edge' \in \edges(K) \right \rbrace,
\end{equation}
for $K \in \mesh$, $\edge \in \edges(K)$, $\epsilon \in \edgesd(D_\sigma)$ with $\edged\subset K$.
\end{itemize}
The system of equations \eqref{eq:F_syst} only depends on the type of the cell $K$ (since it only depends on the coefficient $\xi_K^\edge$, which is chosen as the inverse of the number of cell faces, and sub-cell connectivities) but has an infinite number of solutions, which makes necessary to impose in addition the constraint \eqref{eq:F_bounded}; however, assumptions (H1)-(H3) are sufficient for the subsequent developments, in the sense that any choice for the expression of the fluxes satisfying these assumptions yields stable and consistent schemes (see \cite{lat-18-conv, lat-18-disc}).

\medskip
To complete the definition of the convective flux, we just have now to give the expression of the velocity $\bfv_\edged$ at the dual face. As already said, a dual face lying on the boundary is also a primal face, and the flux across that face is zero.
Therefore, the values $\bfv_\edged$ are only needed at the internal dual faces; we choose them to be centered:
\[
\bfv_\edged = \frac 1 2 ( \bfv_\edge + \bfv_{\edge'}), \qquad \mbox{for } \edged=D_\edge|D_\edge'.
\]
%
% ---------------
%
\subsection{Diffusion term}

The space discretization of the diffusion term $\divv (\bftau(\bfu))$ in the momentum balance equation relies on the Crouzeix-Raviart element for the simplicial cells $K$ and on the \emph{parametric} Rannacher-Turek (or rotated bilinear) element for quadrangular or hexahedral cells  (see \cite{ran-92-sim}).
Let $\mathbb{P}_K$ be the affine transformation between the reference unit simplex and the simplicial cell $K$.
The space of discrete functions over a simplicial cell $K$ is
\[
 P_1(K)= \Bigl\lbrace f\circ \mathbb{P}_K^{-1},\, \text{with}\, f\in {\rm span}\, \bigl\lbrace 1,\,(x_{i})_{i=1,\ldots,d}\bigr\rbrace\Bigr \rbrace.
\]
Let $\mathbb{Q}_K$ be the standard $Q_1$ mapping between the reference unit cuboid and the quadrangular or hexahedral cell $K$.
The space of discrete functions over such a cell is 
\[
\tilde{Q}_1(K)= \Bigl\lbrace f\circ \mathbb{Q}_K^{-1},\,\text{with}\,f\in 
{\rm span}\, \bigl\lbrace 1,\,(x_{i})_{i=1,\ldots,d},\,( x_{i}^2- x_{i+1}^2)_{i=1,\ldots,d-1}\bigr\rbrace\Bigr\rbrace.
\]
The shape functions are the functions $\zeta_\edge,\,\edge\in\edges$ such that for all $K\in\mesh$, $\zeta_\edge|_K\in P_1(K)$ if $K$ is a simplicial cell and $\zeta_\edge|_K\in \tilde{Q}_1(K)$ if $K$ is a quadrangular or hexahedral cell and which satisfy the following two conditions: 
\begin{subequations}
\begin{align}
\label{eq:def_saut_EF} (i) & \quad 
\int_\edge [\zeta]_\edge(\bfx) = 0, \quad\text{where} \ \
[\zeta]_\edge(\bfx) = \lim_{\substack{\bfy \to \bfx \\ \bfy\in L}} \zeta(\bfy)-
\lim_{\substack{\bfy \to \bfx \\ \bfy\in K}} \zeta(\bfy),
\,  \forall\bfx \in \edge, \,\forall\edge\in\edgesint, \,\edge=K|L. \\[1ex]
\label{eq:def_zeta} (ii) &  \quad
\dfrac{1}{|\edge'|}\int_{\edge'} \zeta_\edge(\bfx) = \delta_\edge^{\edge'},
\quad  \forall\edge, \edge' \in \edges,
\end{align}
\end{subequations}
with $\delta_\edge^{\edge'}=1$ if $\edge=\edge'$ and $\delta_\edge^{\edge'}=0$ otherwise.
Note that condition $(i)$ is consistent with a location of the velocity degrees of freedom at the faces.

\medskip
Now, with a discrete velocity field $\bfu\in\xbfH_{\edges,0}(\Omega)$, one classically associates, in the finite element context, the function $\hat\bfu(\bfx)=\sum_{\edge\in\edges}\bfu_\edge\zeta_\edge(\bfx)$. 
The discretization of the diffusion term is a piecewise constant function on each diamond cell $D_\edge$, the value of which reads
\begin{equation}\label{eq:def_diff}
\divv (\bftau(\bfu))_\edge = 
 - \mu \Big ( \frac{1}{|D_\edge|}\sum_{K \in \mesh} \int_K \gradi\hat\bfu \, . \gradi \zeta_\edge \Big ) 
- (\mu + \lambda) \Big (\frac{1}{|D_\edge|} \sum_{K \in \mesh} \int_K \dive(\hat\bfu) \gradi \zeta_\edge \Big ).
\end{equation}
The identification between $\bfu\in\xbfH_{\edges,0}(\Omega)$ and $\hat \bfu$ allows to introduce the broken Sobolev $\xH^1$ semi-norm $\norm{.}_\brok$, given for any $\bfu \in \xbfH_{\edges}(\Omega)$ by:
\[
\norm{\bfu}_\brok^2=\sum_{K\in \mesh} \int_K \gradi \hat{\bfu}:\gradi \hat{\bfu}.
\]
The semi-norm $\norm{\bfu}_\brok$ is in fact a norm on the space $\xbfH_{\edges,0}(\Omega)$, thanks to a classical discrete Poincar\'e inequality.

\medskip
As in the continuous setting, it is easily seen that the bilinear form derived from the discretization of the diffusion term controls the discrete $\xH^1$-norm of the velocity as stated in the following lemma:
\begin{lemma}[Coercivity of the diffusion operator] \label{lmm:diff_coerc}
For every discrete velocity field $\bfu\in\xbfH_{\edges,0}(\Omega)$, one has:
\[
\sum_{\edges \in \edgesint} |D_\edge|\ \bfu_\edge \cdot \big( - \divv (\bftau(\bfu))_\edge \big ) \geq  \mu\ \norm{\bfu}_\brok^2.
\]
\end{lemma}
%
% ---------------
%
\subsection{Pressure gradient term}

The discretization of the pressure gradient term $\gradi \eos(\rho)$ is a piecewise constant function on each diamond cell $D_\edge$, the value of which is denoted $(\gradi p)_\edge$.
For $\rho\in\xL_\mesh(\Omega)$, this term is defined as:
\begin{equation}\label{eq:def_grad_p}
(\gradi p)_\edge=\frac{|\edge|}{|D_\edge|} \ (\eos(\rho_L)-\eos(\rho_K))\ \bfn_{K,\edge}, \qquad \forall\edge=K|L \in \edgesint.
\end{equation}
This pressure gradient is only defined at internal faces since, thanks to the impermeability boundary conditions, no momentum balance equation is written at the external faces. 

\medskip
The following discrete duality relation holds for all $\rho\in\xL_\mesh(\Omega)$ and $\bfu \in \xbfH_{\edges,0}(\Omega)$:
\begin{equation} \label{eq:grad-div}
\sum_{K \in \mesh} |K|\ \eos(\rho_K)\ \dive (\bfu)_K
+\sum_{\edge\in\edgesint} |D_\edge|\ \bfu_\edge \cdot (\gradi p)_\edge
=0,
\end{equation}
where we have set for all $K\in\mesh$, $\dive (\bfu)_K = |K|^{-1}\sum_{\edge\in\edges(K)}|\edge|\, \bfu_\edge\cdot\bfn_{K,\edge}$ (consistently with \eqref{div_mass} for $\rho\equiv 1$).

\bigskip
We finish this section with the following lemma which states that the staggered approximation is \textit{inf-sup} stable.
\begin{lemma} \label{lmm:inf-sup}
There exists $\beta>0$, depending only on $\Omega$ and on the mesh, such that for all $p=\lbrace p_K,\,K\in\mesh\rbrace\in \xL_\mesh(\Omega)$, there exists $\bfu\in\xbfH_{\edges,0} $ satisfying:
\[
\norm{\bfu}_\brok=1 \text{ and } \sum_{K\in\mesh} |K|\ p_K\,\dive (\bfu )_K \geq \beta\, \norm{p-m(p)}_{L^2(\Omega)},
\]
where $m(p)=|\Omega|^{-1}\sum_{K\in\mesh}|K|p_K$ is the mean value of $p$ over $\Omega$.
\end{lemma}

The \textit{inf-sup} property is crucial when passing to the limit $\ma\to 0$ in the various schemes presented thereafter.
It indeed provides an $\xL^2$ control on the discrete (zero-mean) pressure through the control of its gradient.

\medskip
In fact, the actual \textit{inf-sup} stability condition states that the constant $\beta$ only depends on the regularity of the mesh (in a sense to be defined), and not on the space step; this property, which is satisfied by low-order staggered discretizations, is inherited by the limit incompressible scheme and guarantees its stability and the fact that error estimates do not blow up when the mesh is refined.
In this paper, since we work on a fixed discretization, the dependency of $\beta$ with respect to the mesh does not need to be precisely stated.    
%
% -----------------------------------------------------------------------------------
%
\section{Asymptotic analysis of the zero Mach limit for an implicit scheme}\label{sec:impl}

We begin with the analysis of the zero Mach limit for a fully implicit scheme.
Let $\delta t>0$ be a constant time step. 
The approximate solution $(\rho^n,\bfu^n)\in\xL_\mesh(\Omega)\times\xbfH_{\edges,0}(\Omega)$ at time $t_n=n\delta t$ for $1\leq n\leq N=\Ent{T/\delta t}$ is computed by induction through the following implicit scheme. 

\medskip
Knowing $(\rho^n,\bfu^n)\in\xL_\mesh(\Omega)\times\xbfH_{\edges,0}(\Omega)$, solve for $\rho^{n+1}\in\xL_\mesh(\Omega)$ and $\bfu^{n+1}\in\xbfH_{\edges,0}(\Omega)$:
\begin{subequations}\label{eq:implicit_scheme}
\begin{align}
\label{eq:isch_mass}  & 
\dfrac 1 {\delta t}(\rho^{n+1}_K-\rho^n_K) + \dive(\rho^{n+1} \bfu^{n+1})_K = 0, & \forall K \in \mesh,
\\[2ex]  
&  \dfrac 1 {\delta t} \bigl(\rho^{n+1}_\Ds \bfu^{n+1}_\edge-\rho^n_\Ds \bfu_\edge^n \bigr)
+ \divv(\rho^{n+1} \bfu^{n+1}\otimes \bfu^{n+1})_\edge
\nonumber \\
\label{eq:isch_mom} & \hspace{35ex} - \divv (\bftau(\bfu^{n+1}))_\edge
+ \dfrac{1}{\ma^2}(\gradi p^{n+1})_\edge
=0, & \forall \edge \in \edgesint.
\end{align}
\end{subequations}
%
% -----------------------------------------
%
\subsection{Initialization of the scheme}

The initial approximations are given by the average of the initial density $\rho_0^\ma$ on the primal cells and the initial velocity $\bfu_0^\ma$ on the dual cells:
\begin{equation}\label{eq:inicond_is}
\begin{array}{ll} 
\displaystyle \rho_K^0 = \frac 1 {|K|} \int_K \rho_0^\ma, & \qquad \forall \, K \in \mesh, \\[4ex]
\displaystyle \bfu_\edge^0 = \frac 1 {|D_\edge|} \int_{D_\edge} \bfu_0^\ma, & \qquad \forall \, \edge \in \edgesint.
\end{array}
\end{equation}

\medskip
Thereafter, we prove that for every $\ma>0$, there exists a solution $(\rho^\ma,\bfu^\ma)$ to the implicit scheme \eqref{eq:implicit_scheme}-\eqref{eq:inicond_is} and that for a fixed discretization, \emph{i.e.} for a fixed mesh and a fixed time step $\delta t$, the solution  $(\rho^\ma,\bfu^\ma)$ converges as $\ma\to 0$ towards the solution of an implicit scheme for the incompressible Navier-Stokes equations, which is stable thanks to the inf-sup condition.

\medskip
\textbf{Assumption on the initial data} -- For the convergence study performed in this section, it is sufficient to assume that the initial data is ill-prepared, in the sense of Inequality \eqref{eq:u0_rho0_ip}.
%
% -----------------------------------------
%
\subsection{A priori estimates} 

We begin with a first lemma which states that the velocity convection operator defined in Section \ref{sec:vel_conv_op} is built so that if a discrete mass conservation equation is satisfied on the cells of the primal mesh (as in \eqref{eq:isch_mass}) - which is consistent with the staggered discretization - then a discrete mass conservation equation is \emph{also} satisfied on each cell of the dual mesh (see \cite{ans-11-anl} for a proof).

\begin{lemma}
\label{lmm:mass_D}
Let two density fields $ \rho^n$, $\rho^{n+1}\in \xL_\mesh(\Omega)$ and a velocity field $\bfu^{n+1} \in \xbfH_{\edges,0}(\Omega)$ satisfying the discrete mass conservation equation \eqref{eq:isch_mass} on every cell of the primal mesh be given.
Then, the dual densities $\lbrace \rho_\Ds^n, \rho_\Ds^{n+1}, \, \edge \in \edges\rbrace$ and the dual fluxes $\lbrace \fluxd(\rho^{n+1},\bfu^{n+1}), \, \edge \in \edgesint, \, \edged \in \edgesd(D_\edge) \rbrace$ satisfy a finite volume discretization of the mass balance \eqref{eq:pb_mass} over the internal dual cells:
\begin{equation}\label{eq:mass_D}
\frac{|D_\edge|}{\delta t} \ (\rho^{n+1}_\Ds-\rho^n_\Ds)
+ \sum_{\edged\in\edgesd(D_\edge)} \fluxd(\rho^{n+1},\bfu^{n+1})=0, \qquad  \forall\edge\in\edgesint.
\end{equation}
\end{lemma}

The following result states that any solution of the implicit scheme satisfies a discrete counterpart to the kinetic energy balance \eqref{eq:pb_ke}; its derivation relies on the previous relation, namely the dual mass balance \eqref{eq:mass_D}.

\begin{lemma}[Discrete kinetic energy balance] \label{lmm:i_ke}
Any solution to the implicit scheme \eqref{eq:implicit_scheme} satisfies the following equality, for all $\edge \in \edgesint$, and $0\leq n \leq N-1$:
\begin{multline} \label{eq:i_ke}
\dfrac 1 {2\delta t} \Big ( \rho_\Ds^{n+1}\,|\bfu_\edge^{n+1}|^2 - \rho_\Ds^n\,|\bfu_\edge^n|^2 \Big )
+ \frac{1}{2|D_\edge|} \sum_{\edged=D_\edge|D_\edge'}\fluxd(\rho^{n+1},\bfu^{n+1}) \,\bfu_\edge^{n+1}\cdot \bfu_{\edge'}^{n+1}
\\
- \divv(\bftau(\bfu^{n+1}))_\edge \cdot  \bfu_\edge^{n+1} 
+ \frac{1}{\ma^2}(\gradi p^{n+1})_\edge \cdot \bfu_\edge^{n+1} + R_\edge^{n+1}=0,
\end{multline}
where $R_\edge^{n+1}= \dfrac{1}{2\delta t} \rho_\Ds^n|\bfu_\edge^{n+1}-\bfu_\edge^n|^2$.
\end{lemma}

\begin{proof}
Let us take the scalar product of the discrete momentum balance equation \eqref{eq:isch_mom} by the corresponding velocity unknown $\bfu_\edge^{n+1}$, which gives the relation $T_\edge^{\rm conv} - \divv(\bftau(\bfu^{n+1}))_\edge \cdot  \bfu_\edge^{n+1} + \frac{1}{\ma^2}(\gradi p^{n+1})_\edge \cdot \bfu_\edge^{n+1}=0$, with:
\[
T_\edge^{\rm conv} =
\Bigl(
\dfrac 1 {\delta t} \bigl(\rho^{n+1}_\Ds \bfu^{n+1}_\edge-\rho^n_\Ds \bfu_\edge^n \bigr)
+ \dfrac 1 {2|D_\edge|}\ \sum_{\substack{\edged \in \edgesd(D_\edge) \\ \edged=\edgeedgeprime}}  \fluxd(\rho^{n+1},\bfu^{n+1})\ (\bfu^{n+1}_\edge+\bfu^{n+1}_{\edge'})
\Bigr)\cdot \bfu_\edge^{n+1}.
\]
Now, using the identity $2\, (\rho|\bfa|^2-\rho^*\bfa\cdot\bfb) = \rho|\bfa|^2 - \rho^*|\bfb|^2 + \rho^* |\bfa-\bfb|^2 + (\rho-\rho^*)|\bfa|^2$ with $\rho=\rho_\Ds^{n+1}$, $\rho^*=\rho_\Ds^n$, $\bfa=\bfu_\edge^{n+1}$ and $\bfb=\bfu_\edge^n$, we obtain
\[\begin{aligned}
T_\edge^{\rm conv}
&
= \dfrac 1 {2\delta t} \Bigl( \rho^{n+1}_\Ds |\bfu^{n+1}_\edge|^2-\rho^n_\Ds |\bfu_\edge^n|^2 \Bigr)
+ \dfrac 1 {2 |D_\edge|}\ \sum_{\edged=\edgeedgeprime} \fluxd(\rho^{n+1},\bfu^{n+1})\ \bfu^{n+1}_\edge \cdot \bfu^{n+1}_{\edge'}
\\ &
+ \dfrac 1 {2\delta t}\ \rho^n_\Ds\ | \bfu^{n+1}_\edge - \bfu^n_\edge |^2
+ \Bigl(\frac 1 {\delta t} \ (\rho^{n+1}_\Ds-\rho^n_\Ds)
+ |D_\edge|^{-1}\sum_{\edged\in\edges(D_\edge)} \fluxd(\rho^{n+1},\bfu^{n+1}) \Bigr)\, \frac{|\bfu_\edge^{n+1}|^2} 2.
\end{aligned}\]
The last term is equal to zero thanks to \eqref{eq:mass_D}, which concludes the proof.
\end{proof}

\medskip
We then prove that any solution of the implicit scheme satisfies a discrete counterpart of the renormalization identities \eqref{eq:pb_renorm} and \eqref{eq:pb_renorm:Pi} satisfied by any smooth solution of \eqref{eq:pb}.
To state this result, we need to extend the notation for divergence operators on the primal cells as follows.
For $K \in \mesh$ and a smooth function $\varphi$,
\[
\dive\big(\varphi(\rho)\,\bfu\big)_K = \frac 1 {|K|} \sum_{\edge \in \edges(K)} |\edge|\ \varphi(\rho_\edge)\,\bfu_\edge \cdot \bfn_{K,\edge},
\]
where $\rho_\edge$ stands for the upwind value of the density at the face.

\begin{lemma}[Discrete renormalization identities] \label{lem:i_renorm}
Define the function $\psi_\gamma$ as 
$\psi_\gamma(\rho)= \rho\int^\rho_0\frac{\eos(s)}{s^2} ds $, which yields $\psi_\gamma(\rho)=\rho\log\rho$ if $\gamma=1$ and $\psi_\gamma(\rho)=\rho^\gamma/(\gamma-1)$ if $\gamma>1$, and define $\Pi_\gamma(\rho)=\psi_\gamma(\rho)-\psi_\gamma(1)-\psi_\gamma'(1)(\rho-1)$. Then, any solution to the implicit scheme \eqref{eq:implicit_scheme} satisfies the following two identities, for all $K\in\mesh$ and $0\leq n \leq N-1$:
\begin{align} & \label{eq:i_renorm}
\dfrac 1 {\delta t} \Big ( \psi_\gamma(\rho^{n+1}_K)-\psi_\gamma(\rho^n_K) \Big )
+ \dive(\psi_\gamma(\rho^{n+1}) \bfu^{n+1})_K
+ \eos(\rho_K^{n+1}) \, \dive(\bfu^{n+1})_K + R_K^{n+1} = 0,
\\[2ex] & \nonumber
\dfrac 1 {\delta t} \Big ( \Pi_\gamma(\rho^{n+1}_K)-\Pi_\gamma(\rho^n_K) \Big ) 
+ \dive\Big(\big(\psi_\gamma(\rho^{n+1}) - \psi_\gamma'(1)\,\rho^{n+1}\big)\, \bfu^{n+1} \Big)_K
\\ & \label{eq:i_renorm:Pi} \hspace{50ex}
+  \eos(\rho_K^{n+1}) \, \dive( \bfu^{n+1})_K + R_K^{n+1} = 0,
\end{align}
with :
\begin{multline*}
R_K^{n+1} = \frac{1}{2 \delta t} \, \psi_\gamma''(\bar\rho_K^{n+\frac 1 2})\, (\rho_K^{n+1}-\rho_K^n)^2
+\frac{1}{2 |K|} \sum_{\edge = K|L} |\edge| \, (\bfu_\edge^{n+1}\cdot \bfn_{K,\edge})^- \, \psi_\gamma''(\bar \rho_\edge^{n+1}) \, (\rho_L^{n+1}-\rho_K^{n+1})^2,
\end{multline*}
where $\bar \rho_K^{n+\frac 1 2}\in[\min(\rho_K^{n+1},\rho_K^n),\max(\rho_K^{n+1},\rho_K^n)]$, $\bar \rho_\edge^{n+1}\in[\min(\rho_\edge^{n+1},\rho_K^{n+1}),\max(\rho_\edge^{n+1},\rho_K^{n+1})]$ for all $\edge\in\edges(K)$, and for $a\in\xR$, $a^-\geq 0$ is defined by $a^-=-\min(a,0)$. Since $\psi_\gamma$ is a convex function, $R_K^{n+1}$ is non-negative.
\end{lemma}

\begin{proof}
The proof of \eqref{eq:i_renorm} is the same as that of \cite[Lemma 3.2]{her-14-ons}.
The equality  \eqref{eq:i_renorm:Pi} is obtained after multiplying the discrete mass equation \eqref{eq:isch_mass} by $-\psi_\gamma'(1)$ and summing with \eqref{eq:i_renorm}.
\end{proof}

\medskip
As a consequence of Lemmas \ref{lmm:i_ke} and \ref{lem:i_renorm}, the implicit scheme satisfies a discrete counterpart of the local-in-time entropy identity \eqref{entrop_2}.

\begin{lemma}[Local-in-time discrete entropy inequality, existence of a solution] \label{lmm:loc_ener_i}
Let $\ma>0$ and assume that the initial density $\rho_0^\ma$ is positive. Then, there exists a solution $(\rho^n,\bfu^n)_{0\leq n\leq N}$ to the scheme \eqref{eq:implicit_scheme}, such that $\rho^n >0$ for $0\leq n\leq N$, and the following inequality holds for $0\leq n\leq N-1$:
\begin{equation} \label{i_loc_ener}
\frac 1 2  \sum_{\edge \in \edgesint}
|D_\edge| \Bigl( \rho^{n+1}_\Ds |\bfu^{n+1}_\edge|^2-\rho^n_\Ds |\bfu_\edge^n|^2 \Bigr)
+ \frac{1}{\ma^2}\sum_{K \in \mesh} |K|\, (\Pi_\gamma(\rho^{n+1}_K) - \Pi_\gamma(\rho^n_K))
 + \ \mu\,  \delta t\ \norm{\bfu^{n+1}}_{1,\disc}^2
 + \mathcal{R}^{n+1} \leq 0,
\end{equation}
where $  \mathcal{R}^{n+1} = \sum_{\edge \in \edgesint} R^{n+1}_\edge + {\ma}^{-2}\sum_{K \in \mesh} R_K^{n+1}\geq 0$.
\end{lemma}

\begin{proof}
The positivity of the density is a consequence of the properties of the upwind choice \eqref{eq:def_FKedge} for $\rho$ \cite[Lemma 2.1]{gas-11-dis}.
Let us then sum equation \eqref{eq:i_ke} over the faces $\edge \in \edgesint$, ${\ma}^{-2}\times$ \eqref{eq:i_renorm:Pi} over $K \in \mesh$, and, finally, the two obtained relations.
Since the discrete gradient and divergence operators are dual with respect to the $\xL^2$ inner product (see \eqref{eq:grad-div}), noting that the conservative dual fluxes vanish in the summation and that the diffusion term is coercive (see Lemma \ref{lmm:diff_coerc}), we get \eqref{i_loc_ener}.

Given discrete density and velocity fields $(\rho^n,\bfu^n)$, the existence of a solution $(\rho^{n+1},\bfu^{n+1})$ to the implicit scheme at the time step $n+1$ may be inferred by the Brouwer fixed point theorem, by an easy adaptation of the proof of \cite[Proposition 5.2]{eym-10-conv}.
This proof relies on the following set of mesh-dependent estimates: the conservativity of the mass balance discretization, together with the fact that the density is positive, yields an estimate for $\rho$ in the $\xL^1$-norm, and so, by a norm equivalence argument, of the pressure in any norm; then, for a given density $\rho$, the discrete global kinetic energy inequality (\ie\ \eqref{eq:i_ke} summed over the control volumes) provides a control on the velocity.
Therefore, computing $\rho$ from the mass balance for fixed $\bfu$, then $p$ from $\rho$ by the equation of state $\eos(\rho)$, and finally $\bfu$ from the momentum balance equation with fixed $\rho$, yields an iteration in a bounded convex subset of a finite dimensional space.
\end{proof}

\medskip
We may now prove that the solution of the implicit scheme satisfies a discrete counterpart of the global estimate \eqref{estimate_Pi}.

\begin{lemma}[Global discrete entropy inequality]
Let $\ma>0$ and assume that the initial data $(\rho_0^\ma,\bfu_0^\ma)$ is ill-prepared in the sense of \eqref{eq:u0_rho0_ip}. By Lemma \ref{lmm:loc_ener_i}, there exists a solution $(\rho^n,\bfu^n)_{0\leq n\leq N}$ to the scheme \eqref{eq:implicit_scheme}. In addition, there exists $C>0$ independent of $\ma$ such that, for $\ma$ small enough and for all $1\leq n\leq N$:
\begin{equation}
\label{eq:sch_stab}
 \frac 1 2  \sum_{\edge\in \edgesint} |D_\edge| \ \rho^n_\Ds\ |\bfu^n_\edge|^2
+\mu\  \sum_{k=1}^n \delta t\ \norm{\bfu^k}_{1,\disc}^2
+\frac 1 {\ma^2} \sum_{K\in \mesh}|K| \, \Pi_\gamma(\rho^n_K) \leq C.
\end{equation}
\end{lemma}

\begin{proof}
Multiplying equation \eqref{i_loc_ener} by $\delta t$ and summing over the time steps yields for $1\leq n\leq N$:
\begin{multline}\label{eq:imp_stab}
\frac 1 2  \sum_{\edge\in \edgesint} |D_\edge| \ \rho^n_\Ds\ |\bfu^n_\edge|^2
+\mu\  \sum_{k=1}^n \delta t\ \norm{\bfu^k}_{1,\disc}^2
+\frac 1 {\ma^2} \sum_{K\in \mesh}|K| \, \Pi_\gamma(\rho^n_K)+\mathcal{R}^n
\\
\leq \ 
\frac 1 2  \sum_{\edge\in \edgesint} |D_\edge| \ \rho^0_\Ds\ |\bfu^0_\edge|^2
+\frac 1 {\ma^2} \sum_{K\in \mesh}|K| \ \Pi_\gamma(\rho^0_K),
\end{multline}
with $\mathcal{R}^n= \sum_{k=0}^{n-1} \big (\sum_{\edge \in \edgesint} R^{k+1}_\edge + \ma^{-2}\sum_{K \in \mesh} R_K^{k+1} \big ) \geq 0$.
\smallskip
 
Let us prove that the right hand side of \eqref{eq:imp_stab} is uniformly bounded for all $\ma$ small enough. By \eqref{eq:u0_rho0_ip} and \eqref{eq:inicond_is}, for $\ma$ small enough, one has $\rho_K^0 \leq 2$ for all $K\in\mesh$ and therefore  $\rho^0_\Ds\leq 2$ for all $\edge\in\edgesint$, since the dual densities are convex combinations of the primal density unknowns. Hence, one has:
\[
 \frac 1 2 \sum_{\edge\in \edgesint} |D_\edge| \ \rho^0_\Ds\ |\bfu^0_\edge|^2 \leq
\sum_{\edge\in \edgesint} |D_\edge|^{-1} \Big | \int_{D_\edge} \bfu_0^\ma \Big |^2 \leq \norm{\bfu_0^\ma}_{\xL^2(\Omega)^d}^2,
\]
which by \eqref{eq:u0_rho0_ip} is uniformly bounded with respect to $\ma$. Then, using the upper bound on $\Pi_\gamma(\rho)$ for small values of $\rho$ (see Lemma \ref{lem:pi}), we can see that the second term of the right hand side of \eqref{eq:imp_stab} is uniformly bounded with respect to $\ma$, for $\ma$ small enough.
\end{proof}

\begin{lemma}[Control of the pressure] \label{lmm:icontrole:p}
Let $\ma>0$ and assume that the initial density $\rho_0^\ma$ is positive and that the initial data $(\rho_0^\ma,\bfu_0^\ma)$ is ill-prepared in the sense of \eqref{eq:u0_rho0_ip}. Then, there exists a solution $(\rho^n,\bfu^n)_{0\leq n\leq N}$ to the scheme \eqref{eq:implicit_scheme}. Let  $p^n=\eos(\rho^n)$ and define $\deltap^n = \lbrace \deltap_K^n,\,K\in\mesh \rbrace $ where $\deltap_K^n = (p_K^n-m(p^n))/\ma^2$ with $m(p^n)$ the mean value of $p^n$ over $\Omega$ (\emph{i.e.} $m(p^n) = |\Omega|^{-1}\sum_{K\in\mesh} |K|\ p_K^n$). Then, one has, for all $1\leq n \leq N$:
\begin{equation}
\label{eq:icontrole:p}
 \norm{\deltap^n} \leq C_{\disc,\delta t},
\end{equation}
where  the real number $C_{\disc,\delta t}$ depends on the mesh and the time step but not on $\ma$, and $\norm{\cdot}$ stands for any norm on the space of discrete functions.
\end{lemma}

\begin{proof}
The staggered discretization satisfies the \textit{inf-sup} condition (see Lemma \ref{lmm:inf-sup}), which implies that there exists a positive real number $\beta$, depending only on $\Omega$ and on the mesh, such that for $\deltap^n= (p^n-m(p^n))/\ma^2$, there exists a discrete velocity field $\bfv$, with $\norm{\bfv}_\brok=1$ and satisfying:
\[
\beta\, \norm{\deltap^n}_{L^2(\Omega)} \leq \sum_{K\in\mesh} |K|\ p_K^n\,\dive (\bfv )_K.
\]
Hence by the gradient-divergence duality property \eqref{eq:grad-div}, taking the scalar product of \eqref{eq:isch_mom} with $|D_\edge|\,\bfv_\edge$ and summing over $\edge$ in $\edgesint$, we get $\beta\norm{\deltap^n}_{L^2(\Omega)} \leq T^n_1+T^n_2+T^n_3$ with:
\[ \begin{aligned} &
T^n_1 =\dfrac 1 {\delta t}\,\sum_{\edge\in\edgesint}\,|D_\edge|\,(\rho^n_\Ds \bfu^n_\edge-\rho^{n-1}_\Ds \bfu_\edge^{n-1})\cdot\bfv_\edge,
\\ &
T^n_2 =\sum_{\edge\in\edgesint}\,|D_\edge| \, \divv(\rho^n \bfu^n \otimes \bfu^n)_\edge\cdot \bfv_\edge,
\\ &
T^n_3 =- \sum_{\edge\in\edgesint}\,|D_\edge| \, \divv(\bftau(\bfu^n))_\edge \cdot \bfv_\edge.
\end{aligned} \]
The proof of \eqref{eq:icontrole:p} follows if one is able to control each of these terms $T^n_1$, $T^n_2$ and $T^n_3$ independently of $\ma$. Since the mesh and the time step $\delta t$ are fixed, a norm equivalence argument in a finite dimensional space yields the existence of a positive function $C_{\disc,\delta t}$ depending on the discretization, which is non-decreasing in each of its variables, such that: 
\[
|T^n_1+T^n_2+T^n_3| \leq C_{\disc,\delta t} \,(\norm{\rho^n}_{\xL^1},\norm{\rho^{n-1}}_{\xL^1},\norm{\bfu^n}_\brok, \norm{\bfu^{n-1}}_\brok, \norm{\bfv}_\brok).
\]
We have  $\norm{\bfv}_\brok=1$ and by the estimate \eqref{eq:sch_stab}, $\norm{\bfu^n}_\brok$ and $\norm{\bfu^{n-1}}_\brok$ are controlled independently of $\ma$. In addition $\rho^n$ and $\rho^{n-1}$ are controlled in $\xL^1$ by conservativity of the discrete mass balance equation.
\end{proof}
%
% -----------------------------
%
\subsection{Incompressible limit of the implicit scheme}

We may now state the main result of this section which is the convergence, up to a subsequence, of the solution to the compressible scheme \eqref{eq:implicit_scheme} towards the solution of an implicit \emph{inf-sup} stable incompressible scheme when the Mach number tends to zero.

\begin{theorem}[Asymptotic behavior of the implicit scheme]
\label{Theorem:imp}
Let $(\ma\exm)_{m\in \xN}$ be a sequence of positive real numbers tending to zero, and let $(\rho\exm,\bfu\exm)_{m \in \xN}$ be a corresponding sequence of solutions of the scheme \eqref{eq:implicit_scheme}.
Let us assume that the initial data $(\rho_0^{\ma\exm},\bfu_0^{\ma\exm})$ is ill-prepared, \ie\ satisfies Relation \eqref{eq:u0_rho0_ip} for $m \in \xN$.
Then the sequence $(\rho\exm)_{m \in \xN}$ tends to the constant function $\rho =1$ when $m$ tends to $+ \infty$ in $\xL^\infty((0,T),\xL^\gamma(\Omega))$. Moreover, for all $q\in[1,\min(2,\gamma)]$, there exists $C>0$ such that:
\[
   \norm{\rho\exm-1}_{\xL^\infty((0,T);\xL^q(\Omega))} \leq C\ma\exm, \qquad \text{for $m$ large enough}.
\]

\medskip
In addition, the sequences $(\bfu\exm)_{m \in \xN}$ and $(\deltap\exm)_{m \in \xN}$ are bounded in any discrete norm which may depend on the fixed discretization.  If a subsequence of $(\bfu\exm,\deltap\exm)_{m \in \xN}$ tends, in any discrete norm, to a limit $(\bfu,\deltap)$, then $(\bfu,\deltap)$ is a solution to the standard (Rannacher-Turek or Crouzeix-Raviart) implicit scheme for the incompressible Navier-Stokes equations: 

\medskip
Knowing  $\deltap^n\in\xL_\mesh(\Omega)$ and $\bfu^n\in\xbfH_{\edges,0}(\Omega)$, solve for $\deltap^{n+1}\in\xL_\mesh(\Omega)$ and $\bfu^{n+1}\in\xbfH_{\edges,0}(\Omega)$:
\begin{subequations}\label{eq:iinc_scheme}
\begin{align}
\label{eq:iisch_mass}  & 
\dive(\bfu^{n+1})_K = 0, & \forall K \in \mesh,
\\[2ex]  \label{eq:iisch_mom} & 
\dfrac 1 {\delta t} \bigl( \bfu^{n+1}_\edge - \bfu_\edge^n \bigr)
+ \divv( \bfu^{n+1}\otimes  \bfu^{n+1})_\edge
-\divv (\bftau(\bfu^{n+1}))_\edge
+(\gradi {\deltap}^{n+1})_\edge
=0, & \forall \edge \in \edgesint.
\end{align}
\end{subequations}
At the limit, the scheme uses as initial condition only the $\xL^2$-projection of $ \bfu_0$ (the limit of $\bfu_0^\ma$) on the space $E_\disc(\Omega)$ of the discrete divergence-free functions:
$E_\disc(\Omega)=\lbrace \bfv \in \xbfH_{\edges,0}(\Omega),\, \dive(\bfv)_K=0,\, \forall K\in\mesh\rbrace$.
\end{theorem}

\begin{proof}
The proof of the convergence of $(\rho\exm)_{m \in \xN}$ towards $ \rho =1$ when $m$ tends to $+ \infty$ is the same as in the continuous setting. It follows from the combination of the lower bounds on $\Pi_\gamma(\rho)$ proven in Lemma \ref{lem:pi} and the discrete global entropy estimate \eqref{eq:sch_stab}.

Using again estimate \eqref{eq:sch_stab}, one can see that the sequence $(\bfu\exm)_{m \in \xN}$ is bounded in any discrete norm and the same holds for the sequence $(\deltap\exm)_{m \in \xN}$ by Lemma \ref{lmm:icontrole:p}.
By the Bolzano-Weierstrass Theorem and a norm equivalence argument for finite dimensional spaces, there exists a subsequence of  $(\bfu\exm,\deltap\exm)_{m \in \xN}$ which tends, in any discrete norm, to a limit $( \bfu,\deltap)$.
Passing to the limit cell-by-cell in the implicit scheme \eqref{eq:implicit_scheme}, one obtains that $(\bfu, \deltap)$ is a solution of the standard implicit scheme \eqref{eq:iinc_scheme} for the incompressible Navier-Stokes equations.

If $\bfu^1$ is a solution of the first time-step of \eqref{eq:iinc_scheme} then $\bfu^1\in E_\disc(\Omega)$, and for all $\bfvphi\in E_\disc(\Omega)$, taking the scalar product of \eqref{eq:iisch_mom} with $|D_\edge|\bfvphi_\edge$ and summing over $\edge\in\edgesint$, one gets that $\bfu^1$ satisfies:
\begin{equation} \label{imp:stokes}
\frac{1}{\delta t} \bigl ( \bfu^1,\bfvphi \bigr )_{\xL^2} 
+ \sum_{\edge\in\edgesint} |D_\edge|\,\divv(\rho^{1}\bfu^1\otimes\bfu^1)_\edge\cdot\bfvphi_\edge
+ \sum_{K\in\mesh} \int_K \bftau(\hat{ \bfu}^1):\bftau(\hat \bfvphi)= \frac{1}{\delta t} \bigl ( \bfu^0,\bfvphi \bigr )_{\xL^2}. 
\end{equation}
This system is known to be the part of the algebraic system associated with one time step of the scheme determining the velocity (even if the uniqueness of this unknown is guaranteed only for small time steps), in the sense that the velocity may be computed from \eqref{imp:stokes}, the remaining equations yielding the pressure.
Observing that in the right hand side of \eqref{imp:stokes}, $\bfu^0$ can be replaced by its $\xL^2$-projection $\mathcal{P}(\bfu^0)$ onto $E_\disc(\Omega)$ (since $\bigl (\bfu^0-\mathcal{P}(\bfu^0),\bfvphi \bigr )_{\xL^2}=0$ for all $\bfvphi\in E_\disc(\Omega)$ by definition of $\mathcal{P}(\bfu^0)$), we obtain that only the divergence-free part of the initial velocity is seen by the implicit scheme.
\end{proof}

\begin{remark}[Control of the velocity and extension to the Euler equations]
Estimate \eqref{eq:sch_stab} provides a control on the sequence of discrete velocities $(\bfu\exm)_{m\in\xN}$ in a discrete $\xL^2((0,T);\xH^1_0(\Omega)^d)$-norm provided that $\mu>0$. However, even for the Euler case where $\mu=0$, one can derive a uniform bound on $(\bfu\exm)_{m\in\xN}$ in any discrete norm. Indeed, since the mesh is fixed, $\rho\exm\to 1$ as $m\to+\infty$ means that the discrete density unknowns tend to $1$ in every cell. Hence, by the kinetic energy part of \eqref{eq:sch_stab}, there exists a minimum density $\rho_{{\rm min},\disc}>0$ depending on the fixed mesh such that $\norm{\bfu\exm}_{\xL^2(\Omega)^d}\leq \sqrt{\frac{2C}{\rho_{{\rm min},\disc}}}$ for all $m\in\xN$. Therefore, the result of Theorem \ref{Theorem:imp} is still valid for the Euler equations: the solution to the compressible scheme \eqref{eq:implicit_scheme} (with $\mu=\lambda=0$ and homogeneous Neumann boundary conditions on the velocity) converges, when the Mach number tends to zero, towards the solution of an implicit \emph{inf-sup} stable scheme for the incompressible Euler equations.
 \end{remark}
% -------------------------------------------------------------------------------------------------------------------------------------------------------
%
\section{Asymptotic analysis of the zero Mach limit for a pressure correction scheme} \label{sec:proj1}

Since the scheme \eqref{eq:implicit_scheme} is fully implicit, the implementation of the algorithm implies to find the solution of a fully non-linear coupled system which is difficult in a real computational context due to the computational cost and lack of robustness.
This is why we also perform the analysis of the low Mach number limit for a semi-implicit scheme which is implemented in the software (CALIF$^3$S \cite{califs}). This scheme is obtained thanks to a partial decoupling of the discrete equations, and falls in the family of pressure correction schemes.
It consists (after a rescaling step for the pressure gradient) in two main steps. A prediction step, where a tentative velocity field is obtained by solving a linearized momentum balance in which the mass convection flux and the pressure gradient are explicit. A correction step where a nonlinear problem on the pressure is solved, and the velocity is updated in such a way to recover the discrete mass conservation equation. 

\medskip
Let $\delta t>0$ be a constant time step. 
The approximate solution at time $t_n=n\delta t$ for $1\leq n\leq N=\Ent{T/\delta t}$ is denoted $(\rho^n,\bfu^n)\in\xL_\mesh(\Omega)\times\xbfH_{\edges,0}(\Omega)$.

\medskip
Knowing $(\rho^{n-1},\rho^n,\bfu^n)\in\xL_\mesh(\Omega)\times\xL_\mesh(\Omega)\times\xbfH_{\edges,0}(\Omega)$, the considered algorithm consists in computing $\rho^{n+1}\in\xL_\mesh(\Omega)$ and $\bfu^{n+1}\in\xbfH_{\edges,0}(\Omega)$ through the following steps:
\begin{subequations}\label{eq:correction_scheme}
\begin{align}
\nonumber &
\mbox{{\it Pressure gradient scaling step}:}
\\[2ex] \label{eq:sch_scale} & \quad \displaystyle
(\overline{\gradi p})^n_\edge = \Bigl(\frac{\rho^n_\Ds}{\rho^{n-1}_\Ds}\Bigr)^{1/2} \ (\gradi p^n)_\edge, & \forall \edge \in \edgesint.
\\[2ex]\nonumber &
\mbox{{\it Prediction step} -- Solve for $\tilde \bfu^{n+1} \in\xbfH_{\edges,0}(\Omega)$:}
\\[2ex] \label{eq:sch_mom} & \quad
\dfrac 1 {\delta t}\ \bigl(\rho^n_\Ds \tilde \bfu^{n+1}_\edge-\rho^{n-1}_\Ds \bfu_\edge^n \bigr)
+ \divv(\rho^n  \bfu^n \otimes  \tilde \bfu^{n+1} )_\edge
- \divv (\bftau(\tilde \bfu^{n+1}))_\edge
+ \dfrac{1}{\ma^2}\,(\overline {\gradi p})^n_\edge
=0, & \forall \edge \in \edgesint.
\\[3ex] \nonumber &
\mbox{{\it Correction step} -- Solve for $\rho^{n+1}\in\xL_\mesh(\Omega)$ and $\bfu^{n+1}\in\xbfH_{\edges,0}(\Omega)$:}
\\[1ex] \label{eq:sch_cor} & \quad
\dfrac 1 {\delta t}\ \rho^n_\Ds\ (\bfu^{n+1}_\edge-\tilde \bfu_\edge^{n+1})
+ \dfrac{1}{\ma^2}\,(\gradi p^{n+1})_\edge - \dfrac 1 {\ma^2}\,(\overline{\gradi p})_\edge^n
=0, & \forall \edge \in \edgesint.
\\[2ex] \label{eq:sch_mass} & \quad
\dfrac 1 {\delta t}(\rho^{n+1}_K-\rho^n_K) + \dive(\rho^{n+1} \bfu^{n+1})_K = 0, & \forall K \in \mesh.
\end{align}
\end{subequations}
%
% ---------------------------
%
\subsection{Initial data and initialization of the scheme}

As mentioned in the introduction, for the pressure correction scheme, it is not sufficient to assume "ill-prepared" initial data. In this section, the initial data $(\rho_0^\ma,\bfu_0^\ma)$ are assumed to be well-prepared in the sense of the following definition:

\begin{definition}
\label{defi_u0_rho0}
The initial data $(\rho_0^\ma,\bfu_0^\ma)$ is said to be \emph{well-prepared} if $\rho_0^\ma >0$, $\rho_0^\ma \in \xL^\infty(\Omega)$, $\bfu_0^\ma \in \xH_0^1(\Omega)^d$ for all $\ma>0$ and if there exists $C$ independent of $\ma$ such that: 
\begin{equation} \label{eq:u0_rho0_wp}
\norm{\bfu_0^\ma}_{\xH^1(\Omega)^d}  \ + \frac{1}{\ma}\,\norm{\dive \,\bfu_0^\ma}_{\xL^2(\Omega)} \ + \frac{1}{\ma^2}\,\norm{\rho^\ma_0- 1}_{\xL^\infty(\Omega)}  \leq  C.
\end{equation}
Consequently, $\rho_0^\ma$ tends to $1$ when $\ma\to 0$; moreover, we suppose that $\bfu_0^\ma$ converges in $\xL^2(\Omega)^d$ towards a function $\bfu_0\in\xL^2(\Omega)^d$ (the uniform boundedness of the sequence in the $\xH^1(\Omega)^d$ norm already implies this convergence up to a subsequence).
\end{definition}

The initial velocities $\bfu_0^\ma$ are in $\xH_0^1(\Omega)^d$. In particular, their trace is thus well defined as an $\xL^2$ function on any smooth (or Lipschitz-continuous) hypersurface of $\Omega$. The initialization of the pressure correction scheme \eqref{eq:correction_scheme} is performed as follows. First, $\rho^0$ and $\bfu^0$ are given by the average of the initial conditions $\rho_0^\ma$ and $\bfu_0^\ma$ respectively on the primal cells and on the faces of the primal cells:
\begin{equation}\label{eq:inicond_sch}
\begin{array}{ll} 
\displaystyle \rho_K^0 = \frac 1 {|K|} \int_K \rho_0^\ma, & \qquad \forall \, K \in \mesh, \\[4ex]
\displaystyle \bfu_\edge^0 = \frac 1 {|\edge|} \int_\edge \bfu_0^\ma, & \qquad \forall \, \edge \in \edgesint.
\end{array}
\end{equation}
Finally, we compute $\rho^{-1}$ by solving the mass balance equation \eqref{eq:sch_mass} for $n=-1$, where the unknown is $\rho^{-1}$ and not $\rho^0$. 
This procedure allows to perform the first prediction step with $(\rho^{-1}_\Ds)_{\edge \in \edges}$, $(\rho^0_\Ds)_{\edge \in \edges}$ and the dual mass fluxes satisfying the mass balance :
 \begin{equation}\label{eq:mass_D_0}
\frac{|D_\edge|}{\delta t} \ (\rho^0_\Ds-\rho^{-1}_\Ds)
+ \sum_{\edged\in\edgesd(D_\edge)} \fluxd(\rho^0,\bfu^0)=0, \qquad  \forall\edge\in\edgesint.
\end{equation}

\medskip
In this section, we prove that for every $\ma>0$, there exists a solution $(\rho^\ma,\bfu^\ma)$ to the pressure correction scheme \eqref{eq:correction_scheme} with the aforementioned initialization, and that for a fixed discretization, \emph{i.e.} for a fixed mesh and a fixed time step $\delta t$, the solution  $(\rho^\ma,\bfu^\ma)$ converges as $\ma\to 0$ towards the solution of a pressure correction scheme for the incompressible Navier-Stokes equations, which is inf-sup stable. 

\medskip
A first property must be checked so that the scheme is well defined, which is the positivity of the density at the fictitious time step $n=-1$. Indeed,  since $\rho_0^\ma$ is assumed to be positive on $\Omega$ and by definition of the discrete density $\rho^0$ \eqref{eq:inicond_sch}, one clearly has $\rho_K^0 > 0$ for all $K\in\mesh$. This positivity property is not clear for $\rho^{-1}$. It is a consequence of the following result.

\begin{lemma} \label{lmm:sch:init}
If the initial conditions $(\rho_0^\ma,\bfu_0^\ma)$ are assumed to be well-prepared in the sense of Definition \ref{defi_u0_rho0}, then there exists a constant $C$ independent of $\ma$ such that:
\begin{equation}
\label{eq:sch:init}
 \frac{1}{\ma^2} \, \max\limits_{K\in\mesh}|\rho_K^0-1| \ +  \frac{1}{\ma^2} \,  \max\limits_{\edge\in\edgesint}  \,|(\gradi p^0)_\edge| \ + \frac{1}{\ma} \, \max\limits_{K\in\mesh}|\rho_K^{-1}-1| \leq C.
\end{equation}
\end{lemma}

\begin{proof}
Since the initial conditions $(\rho_0^\ma,\bfu_0^\ma)$ are well-prepared, the discrete density at the time step $n=0$ satisfies for all $K\in\mesh$:
\[
 |\rho_K^0-1| \leq  \frac 1 {|K|} \int_K |\rho_0^\ma-1| \leq \, \norm{\rho_0^\ma-1}_{\xL^\infty(K)} \leq C \ma^2.
\]
Then, using the definition of the pressure gradient, we also obtain $|(\gradi p^0)_\edge| \leq C \ma^2$ for all $\edge\in\edgesint$, where the constant $C$ may depend on the (fixed) mesh and on $=\gamma$ (or on $\eos'(1)$ in the general barotropic case) but is independent of $\ma$. Let us then verify that for $\ma$ small enough the discrete density at time step $n=-1$ is close to $1$. This density is computed in such a manner that, for all $K\in\mesh$:
\[
\begin{aligned}
\rho_K^{-1} -1 &= \rho_K^0 -1 + \delta t\, \sum_{\edge\in\edges(K)} \frac{|\edge|}{|K|}\, \rho_\edge^0 \,\bfu_\edge^0\cdot\bfn_{K,\edge} \\
		& = \rho_K^0 -1 + \delta t\, \rho_K^0\, \sum_{\edge\in\edges(K)}\frac{|\edge|}{|K|}\,\,\bfu_\edge^0\cdot\bfn_{K,\edge}  + \delta t\,  \sum_{\edge\in\edges(K)}\frac{|\edge|}{|K|}\, (\rho_\edge^0-\rho_K^0) \,\bfu_\edge^0\cdot\bfn_{K,\edge}\\
		& = \rho_K^0 -1 + \delta t\, \rho_K^0 \frac{1}{|K|} \sum_{\edge\in\edges(K)} \int_\edge \bfu_0^\ma \cdot \bfn_{K,\edge}
                  + \delta t\, \sum_{\edge\in\edges(K)}\frac{|\edge|}{|K|}\, (\rho_\edge^0-\rho_K^0) \,\bfu_\edge^0\cdot\bfn_{K,\edge} \\
		& = \rho_K^0 -1 + \delta t\, \rho_K^0 \frac{1}{|K|}\, \int_K \dive \, \bfu_0^\ma  + \delta t\, \sum_{\edge\in\edges(K)}\frac{|\edge|}{|K|}\, (\rho_\edge^0-\rho_K^0) \,\bfu_\edge^0\cdot\bfn_{K,\edge}.
\end{aligned}
\]
We have already proven that the first term and thus the third term (using, for this latter, a trace inequality to bound the face velocity) in the right-hand side of the above equality are bounded by $C \ma^2$ for some constant $C$ independent of $\ma$ (but dependent on the mesh).
The second term is bounded by $C\ma$ for some constant C since the divergence of the initial velocity satisfies \eqref{eq:u0_rho0_wp}.
The expected bound on $|\rho_K^{-1}-1|$ follows.
\end{proof}

\begin{remark}
All the developments below are still valid under the slightly less restrictive condition on the divergence of the initial velocities:  
\[
 \norm{\dive \, \bfu_0^\ma}_{\xL^2(\Omega)} \to 0 \quad \text{as $\ma\to0$}.
\]
Indeed, under this relaxed condition, we still have $\max\limits_{K\in\mesh}|\rho_K^{-1}-1|\to 0$ as $\ma\to 0$.
\end{remark}
%
% ---------------------------
%
\subsection{A priori estimates}

Let us now derive the estimates satisfied by the solutions of the pressure correction scheme.
As for the implicit scheme, any solution of the pressure correction scheme satisfies a discrete counterpart to the kinetic energy balance \eqref{eq:pb_ke}.

\begin{lemma}[Discrete kinetic energy balance] \label{lmm:corr_ke}
Any solution to the pressure correction scheme \eqref{eq:correction_scheme} satisfies the following equality, for all $\edge \in \edgesint$ and $0\leq n \leq N-1$:
\begin{multline}
\label{eq:sch_ke}
\dfrac 1 {2\delta t} \Big ( \rho_\Ds^n\,|\bfu_\edge^{n+1}|^2 - \rho_\Ds^{n-1}\,|\bfu_\edge^n|^2 \Big ) + \frac{1}{2|D_\edge|} \sum_{\edged=D_\edge|D_\edge'} \hspace{-2ex}\fluxd(\rho^n,\bfu^n) \,\tilde \bfu_\edge^{n+1}\cdot \tilde \bfu_{\edge'}^{n+1}
- \divv(\bftau(\tilde \bfu^{n+1}))_\edge \cdot  \tilde \bfu_\edge^{n+1} 
\\
+ \frac{1}{\ma^2}(\gradi p^{n+1})_\edge \cdot \bfu_\edge^{n+1}  
+ \frac{\delta t}{2\,\ma^4} \ \Bigl( \frac{\bigl | (\gradi p^{n+1})_\edge \bigr |^2}{\rho^n_\Ds}   -\frac{\bigl | (\gradi p^n)_\edge \bigr |^2}{\rho^{n-1}_\Ds}\Bigr)+ R_\edge^{n+1}=0, 
\end{multline}
where $R_\edge^{n+1}= \dfrac{1}{2\delta t} \rho_\Ds^{n-1}|\tilde \bfu_\edge^{n+1}-\bfu_\edge^n|^2$.
\end{lemma}

\begin{proof}
Let us take the scalar product of the velocity prediction equation \eqref{eq:sch_mom} with the corresponding velocity unknown $\tilde \bfu_\edge^{n+1}$.
Thanks to the time shift in the dual densities and the dual mass fluxes in the velocity prediction equation \eqref{eq:sch_mom}, together with the dual mass balance (at the previous time step) 
\[
\frac 1 {\delta t} \ (\rho^n_\Ds-\rho^{n-1}_\Ds)+ |D_\edge|^{-1}\sum_{\edged\in\edges(D_\edge)} \fluxd(\rho^n,\tilde \bfu^n) = 0,
\qquad \forall \edge\in\edgesint,
\]
we obtain, by a similar computation to that in the proof of Lemma \ref{lmm:i_ke}:
\begin{multline} \label{cvbp:eq:mom_pred}
\dfrac{1}{2\,\delta t} \Bigl( \rho^n_\Ds |\tilde u^{n+1}_\edge|^2-\rho^{n-1}_\Ds |\bfu_\edge^n|^2 \Bigr)
+ \dfrac{1}{2\,|D_\edge|}\ \sum_{\edged=\edgeedgeprime} \fluxd(\rho^n,\tilde \bfu^n)\ \tilde \bfu^{n+1}_\edge \cdot \tilde \bfu^{n+1}_{\edge'} 
 \\ - \divv(\bftau(\tilde\bfu^{n+1}))_\edge \cdot  \tilde \bfu_\edge^{n+1} + \frac{1}{\ma^2}(\overline{\gradi p})_\edge^{n} \cdot \tilde \bfu_\edge^{n+1} + R_\edge^{n+1}
=0,
\end{multline}
where $R_\edge^{n+1}= \dfrac{1}{2\delta t} \rho_\Ds^{n-1}|\tilde \bfu_\edge^{n+1}-\bfu_\edge^n|^2$.
Dividing the velocity correction equation \eqref{eq:sch_cor} by $\Big (\dfrac{\rho^n_\Ds}{\delta t} \Big )^{\frac 1 2} $, we obtain:
\[
\Bigl(\frac{\rho^n_\Ds}{\delta t} \Bigr)^{1/2}\, \bfu_\edge^{n+1}
+ \Bigl( \frac{\delta t }{\rho^n_\Ds}\Bigr)^{1/2}\, \frac{1}{\ma^2}\, (\gradi p^{n+1})_\edge
= \Bigl(\frac{\rho^n_\Ds}{\delta t}  \Bigr)^{1/2}\,\tilde \bfu_\edge^{n+1}
+ \Bigl(\frac{\delta t}{\rho^n_\Ds}\Bigr)^{1/2}\, \frac{1}{\ma^2}\, (\overline{\gradi  p})_\edge^{n}.
\]
Squaring this relation and summing it with \eqref{cvbp:eq:mom_pred} yields
\begin{multline*}
 \dfrac 1 {2\delta t} \Big ( \rho_\Ds^n\,|\bfu_\edge^{n+1}|^2 - \rho_\Ds^{n-1}\,|\bfu_\edge^n|^2 \Big ) + \frac{1}{2|D_\edge|} \sum_{\edged=D_\edge|D_\edge'}\fluxd(\rho^n,\bfu^n) \,\tilde \bfu_\edge^{n+1}\cdot\tilde \bfu_{\edge'}^{n+1} \\  - \divv(\bftau(\tilde \bfu^{n+1}))_\edge \cdot  \tilde \bfu_\edge^{n+1} + \frac{1}{\ma^2}(\gradi p^{n+1})_\edge \cdot \bfu_\edge^{n+1} + R_\edge^{n+1}+ \frac{1}{\ma^4}\,P_\edge^{n+1}=0,
\end{multline*}
where $P_\edge^{n+1} = \dfrac{\delta t}{2\, \rho_\Ds^n}\left ( \bigl | (\gradi p^{n+1})_\edge \bigr |^2- \bigl | (\overline{\gradi p})^{n}_\edge \bigr |^2\right )$.
Recalling that the rescaled pressure is defined as 
\[
(\overline{\gradi p})^{n}_\edge = \Bigl(\frac{\rho^n_\Ds}{\rho^{n-1}_\Ds}\Bigr)^{1/2} \ (\gradi p^n)_\edge
\]
for all $\edge\in\edgesint$, we obtain \eqref{eq:sch_ke}.
\end{proof}

\medskip
The discrete renormalization identities are again valid for the pressure correction algorithm, thanks to the fact that the mass balance \eqref{eq:sch_mass} is satisfied. The proof is identical to that of Lemma \ref{lem:i_renorm} given for the implicit scheme.

\begin{lemma}[Discrete renormalization identities] \label{lem:sch_cor_remorm}
A solution to the system \eqref{eq:correction_scheme} satisfies for all $K\in\mesh$ and $0\leq n \leq N-1$ the two following identities:
\begin{align} & \label{eq:sch_renorm}
\dfrac 1 {\delta t} \Big ( \psi_\gamma(\rho^{n+1}_K)-\psi_\gamma(\rho^n_K) \Big ) + \dive(\psi_\gamma(\rho^{n+1}) \bfu^{n+1})_K 
+  \eos(\rho_K^{n+1}) \, \dive( \bfu^{n+1})_K + R_K^{n+1} = 0, 
\\ & \nonumber
\dfrac 1 {\delta t} \Big ( \Pi_\gamma(\rho^{n+1}_K)-\Pi_\gamma(\rho^n_K) \Big )
+ \dive\Big(\big(\psi_\gamma(\rho^{n+1}) - \psi_\gamma'(1)\,\rho^{n+1}\big)\, \bfu^{n+1} \Big)_K 
\\ & \label{eq:sch_renorm:Pi} \hspace{50ex}
+ \eos(\rho_K^{n+1}) \, \dive( \bfu^{n+1})_K + R_K^{n+1} = 0,
\end{align}
where $R_K^{n+1}$ has the same expression as in Lemma \ref{lem:i_renorm}.
\end{lemma}

\medskip
\begin{lemma}[Local-in-time discrete entropy inequality, existence of a solution] \label{lmm:loc_ener_sch}
Let $\ma>0$ and assume that the initial data $(\rho_0^\ma,\bfu_0^\ma)$ is well-prepared in the sense of Definition \ref{defi_u0_rho0}, \ie\ satisfies \eqref{eq:u0_rho0_wp}.
Then, for $\ma$ small enough to ensure that $\rho^{-1}$ is positive, there exists a solution $(\rho^n,\bfu^n)_{0\leq n\leq N}$ to the scheme \eqref{eq:correction_scheme} such that $\rho^n >0$ for $1\leq n\leq N$ and the following inequality holds for $1\leq n\leq N-1$:
\begin{multline}\label{sch_loc_ener}
\frac 1 2 \sum_{\edge\in \edgesint} |D_\edge| \Big ( \rho^n_\Ds\ |\bfu^{n+1}_\edge|^2 -  \rho^{n-1}_\Ds\ |\bfu^n_\edge|^2 \Big ) \ + \ \frac 1 {\ma^2} \sum_{K\in \mesh}|K| \Big ( \Pi_\gamma(\rho^{n+1}_K) - \Pi_\gamma(\rho^n_K) \Big )
\\
+ \mu\ \delta t\ \norm{\tilde \bfu^{n+1}}_{1,\disc}^2 
+ \frac {\delta t^2} {2\ma^4} \sum_{\edge\in \edgesint}|D_\edge|\,  
\ \Bigl( \frac{\bigl | (\gradi p^{n+1})_\edge \bigr |^2}{\rho^n_\Ds}   -\frac{\bigl | (\gradi p^n)_\edge \bigr |^2}{\rho^{n-1}_\Ds}\Bigr)  
+ \mathcal{R}^{n+1} \leq 0,
\end{multline}
where $\mathcal{R}^{n+1} = \sum_{\edge \in \edgesint} R^{n+1}_\edge + \ma^{-2}\sum_{K \in \mesh} R_K^{n+1}\geq 0$.
\end{lemma}

\begin{proof}
For $\ma$ small enough, by Lemma \ref{lmm:sch:init}, both $\rho^{-1}$ and $\rho^0$ are positive.
The positivity of $\rho^n$, for $1 \leq n \leq N$, is a consequence of the properties of the upwind choice \eqref{eq:def_FKedge} for $\rho$ in the mass balance of the correction step \eqref{eq:sch_mass}.

After multiplication by $|D_\edge|$, we sum the kinetic energy balance equation \eqref{eq:sch_ke} over the faces, and after multiplication by $\ma^{-2}|K|$, we sum the relative entropy balance \eqref{eq:sch_renorm:Pi} over the primal cells, and finally sum the two obtained relations.
Since the discrete gradient and divergence operators are dual with respect to the $\xL^2$ inner product (see \eqref{eq:grad-div}), noting that the conservative fluxes vanish in the summation and that the diffusion term is coercive (see Lemma \ref{lmm:diff_coerc}), we get \eqref{sch_loc_ener}.

Given discrete density and velocity fields $(\rho^{n-1},\rho^n,\bfu^n)$, the predicted velocity $\tilde \bfu^{n+1}$ is the solution of the linear system \eqref{eq:sch_mom}.
Multiplying \eqref{cvbp:eq:mom_pred} by $|D_\edge|$, summing over the faces and using Young's inequality, yields the following inequality, which is valid for all $\alpha >0$:
\begin{equation*}
\frac{1}{2\delta t}\norm{\sqrt{\underline \rho^n}\tilde \bfu^{n+1}}_{\xL^2}^2 
- \frac{\alpha}{2\ma^2}\norm{\tilde \bfu^{n+1}}_{\xL^2}^2
+ \mu\ \norm{\tilde \bfu^{n+1}}_{1,\disc}^2
\leq
\frac{1}{2\delta t}\norm{\sqrt{\underline\rho^{n-1}}\bfu^n}_{\xL^2}^2 + \frac{1}{2\alpha\ma^2}\norm{(\overline{\gradi p})^{n}}_{\xL^2}^2.
\end{equation*}
In this inequality, $\underline\rho^n$ and $\underline\rho^{n-1}$ are the piecewise constant fields equal to $\rho_\Ds^n$ and $\rho_\Ds^{n-1}$ on each dual face $D_\edge$. Choosing $\alpha>0$ small enough yields the coercivity of the linear system \eqref{eq:sch_mom} and therefore the existence of a unique solution $\tilde \bfu^{n+1}$ to the prediction step. The existence of a solution $(\rho^{n+1},\bfu^{n+1})$ to the correction step \eqref{eq:sch_cor}-\eqref{eq:sch_mass} follows from the Brouwer fixed point theorem, by an easy adaptation of the proof of \cite[Proposition 5.2]{eym-10-conv}.
\end{proof}

Let us now turn to the global entropy inequality.

\begin{lemma}[Global discrete entropy inequality]
Let $\ma>0$ and assume that the initial data $(\rho_0^\ma,\bfu_0^\ma)$ is well-prepared in the sense of Definition \ref{defi_u0_rho0}, \ie\ satisfies \eqref{eq:u0_rho0_wp}.
By Lemma \ref{lmm:loc_ener_sch}, there exists a solution $(\rho^n,\bfu^n)_{0\leq n\leq N}$ to the scheme \eqref{eq:implicit_scheme}. 
In addition, there exists $C>0$ independent of $\ma$ such that, for $\ma$ small enough and for all $1\leq n\leq N$:
\begin{equation} \label{eq:corr_stap_bis}
\frac 1 2 \sum_{\edge\in \edgesint} |D_\edge| \ \rho^{n-1}_\Ds\ |\bfu^n_\edge|^2
+ \mu\  \sum_{k=1}^n \delta t\ \norm{\tilde \bfu^k}_{1,\disc}^2
+ \frac 1 {\ma^2} \sum_{K\in \mesh}|K| \, \Pi_\gamma(\rho^n_K)
+ \frac {\delta t^2}{2\ma^4} \sum_{\edge\in\edgesint} \frac{|D_\edge|}{\rho^{n-1}_\Ds}\ |(\gradi p^n)_\edge|^2
\leq C.
\end{equation}
\end{lemma}

\begin{proof}
Summing \eqref{sch_loc_ener} over $n$ yields the expected inequality \eqref{eq:corr_stap_bis} with
\[
 C= \frac 1 2 \sum_{\edge\in \edgesint} |D_\edge| \ \rho^{-1}_\Ds\ |\bfu^0_\edge|^2  +\frac 1 {\ma^2} \sum_{K\in \mesh}|K| \, \Pi_\gamma(\rho^0_K)
+ \frac {\delta t^2}{2\ma^4} \sum_{\edge\in\edgesint} \frac{|D_\edge|}{\rho^{-1}_\Ds}\ |(\gradi p^0)_\edge|^2.
\]
Let us prove that if the initial data is well-prepared in the sense of \eqref{eq:u0_rho0_wp}, then for $\ma$ small enough, $C$ is uniformly bounded independently of $\ma$. 
By Lemma \ref{lmm:sch:init}, $\rho_K^{-1}$ is bounded for all $K\in\mesh$ for $\ma$ small enough and therefore so is $\rho_\Ds^{-1}$ for all $\edge\in\edgesint$. 
Hence, since $\bfu_0^\ma$ is uniformly bounded in $\xH^1(\Omega)^d$ by \eqref{eq:u0_rho0_wp}, a classical trace inequality yields the boundedness of the first term. By \eqref{eq:sch:init}, one has $|\rho_K^0-1|\leq C\ma^2$ for all $K\in\mesh$. 
Hence, by \eqref{upperPi}, the second term vanishes as $\ma\to 0$. The third term is also uniformly bounded with respect to $\ma$ thanks to \eqref{eq:sch:init}.
\end{proof}

\begin{lemma}[Control of the pressure] \label{lmm:sch:controle:p}
Let $\ma>0$ and assume that the initial data $(\rho_0^\ma,\bfu_0^\ma)$ is well-prepared in the sense of Definition \ref{defi_u0_rho0}, \ie\ satisfies \eqref{eq:u0_rho0_wp}.
By Lemma \ref{lmm:loc_ener_sch}, for $\ma$ small enough, there exists a solution $(\rho^n,\bfu^n)_{0\leq n\leq N}$ to the scheme \eqref{eq:implicit_scheme}.
Let $p^n=\eos(\rho^n)$ and define $\deltap^n = \lbrace \deltap_K^n,\,K\in\mesh \rbrace $ where $\deltap_K^n = (p_K^n-m(p^n))/\ma^2$ with $m(p^n)$ the mean value of $p^n$ over $\Omega$ (\emph{i.e.} $m(p^n) = |\Omega|^{-1}\sum_{K\in\mesh} |K|\ p_K^n$).
Then, one has, for $1 \leq n \leq N$:
\begin{equation} \label{eq:sch:controle:p}
\norm{\deltap^n} \leq C_{\disc,\delta t},
\end{equation}
where  the real number $C_{\disc,\delta t}$ depends on the mesh and the time step but not on $\ma$, and $\norm{\cdot}$ stands for any norm on the space of discrete functions.
\end{lemma}

\begin{proof}
According to \eqref{eq:corr_stap_bis}, the discrete pressure gradient is controlled in $\xL^\infty$ by $C_{\disc,\delta t}\, \ma^2$ where $C_{\disc,\delta t}$ is a mesh-dependent constant independent of $\ma$. Hence, by a finite-dimensional argument $\gradi(\delta p^n)$ is controlled in any norm by some constant independent of $\ma$.
In particular, on has 
$\norm{\gradi(\delta p^n)}_{-1,\disc} \leq C_{\disc,\delta t}$, where $\norm{\,.\,}_{-1,\disc}$ is the discrete $(\xH^{-1})^d$-norm defined for $\bfv\in\xbfH_{\edges,0}(\Omega)$ by:
\[
\norm{\bfv}_{-1,\disc} = \sup \Big \lbrace   \sum_{\edge\in\edges} |D_\edge|\ \bfv_\edge \cdot \bfw_\edge,
\quad \text{with } \bfw\in\xbfH_{\edges,0}(\Omega) \text{ such that }\norm{\bfw}_{1,\disc}\leq 1\Big \rbrace.
\]
Invoking the gradient divergence duality \eqref{eq:grad-div} and the \emph{inf-sup} stability of the scheme (see Lemma \ref{lmm:inf-sup}), $\norm{\gradi(\delta p^n)}_{-1,\disc} \leq C_{\disc,\delta t}$ implies that $\norm{\deltap^n}_{\xL^2} \leq \beta^{-1} C_{\disc,\delta t}$.
\end{proof}
%
% ------------------------------------------------
%
\subsection{Incompressible limit of the pressure correction scheme}

We may now state the main result of this section which is the convergence, up to a subsequence, of the solution to the compressible scheme \eqref{eq:correction_scheme} towards the solution of a pressure correction \textit{inf-sup} stable scheme when the Mach number tends to zero.

\begin{theorem}[Incompressible limit of the pressure correction scheme]\ \\
\label{Theorem:esch}
Let $(\ma\exm)_{m\in \xN}$ be a sequence of positive real numbers tending to zero, and let $(\rho\exm,\bfu\exm)$ be a corresponding sequence of solutions of the scheme \eqref{eq:correction_scheme}.
Let us assume that the initial data $(\rho_0^{\ma\exm},\bfu_0^{\ma\exm})$ is well-prepared in the sense of Definition \ref{defi_u0_rho0}, \ie\ satisfies \eqref{eq:u0_rho0_wp} for $m \in \xN$.
Then the sequence $(\rho\exm)_{m \in \xN}$ tends to the constant function $\rho =1$ when $m$ tends to $+ \infty$ in $\xL^\infty((0,T),\xL^\gamma(\Omega))$. Moreover, for all $q\in[1,\min(2,\gamma)]$, there exists $C>0$ such that:
\[
\norm{\rho\exm-1}_{\xL^\infty((0,T);\xL^q(\Omega))} \leq C\ma\exm, \qquad \text{for $m$ large enough}.
\]
In addition, the sequence $(\bfu\exm,\deltap\exm)_{m \in \xN}$ tends, in any discrete norm, to the solution $(\bfu,\deltap)$ of the usual (Rannacher-Turek or Crouzeix-Raviart) pressure correction scheme for the incompressible Navier-Stokes equations, which reads:

\medskip
\begin{minipage}{0.8\textwidth}
Knowing  $\deltap^n\in\xL_\mesh(\Omega)$ and $\bfu^n\in\xbfH_{\edges,0}(\Omega)$, compute $\deltap^{n+1}\in\xL_\mesh(\Omega)$ and $\bfu^{n+1}\in\xbfH_{\edges,0}(\Omega)$ through the following steps:
\end{minipage}
\begin{subequations}\label{eq:limit_correction_scheme}
\begin{align}
\nonumber &
\mbox{{\it Prediction step} -- Solve for $\tilde \bfu^{n+1} \in\xbfH_{\edges,0}(\Omega)$:}
\\[2ex] \label{eq:limit_sch_mom} & \quad
\dfrac 1 {\delta t}\ \bigl(\tilde \bfu^{n+1}_\edge-\bfu_\edge^n \bigr)
+ \divv(\bfu^n \otimes  \tilde \bfu^{n+1} )_\edge
- \divv (\bftau(\tilde \bfu^{n+1}))_\edge
+ (\gradi \deltap^n)_\edge
=0, & \forall \edge \in \edgesint.
\\[3ex] \nonumber &
\mbox{{\it Correction step} -- Solve for $\deltap^{n+1}\in\xL_\mesh(\Omega)$ and $\bfu^{n+1}\in\xbfH_{\edges,0}(\Omega)$:}
\\[1ex]  \label{eq:limit_sch_cor}  
& \quad
\dfrac 1 {\delta t}\  (\bfu^{n+1}_\edge-\tilde \bfu_\edge^{n+1})
+ (\gradi \deltap^{n+1})_\edge- (\gradi \deltap^n)_\edge
=0, & \forall \edge \in \edgesint.
\\[2ex] \label{eq:limit_sch_mass} & \quad
\dive(\bfu^{n+1})_K = 0, & \forall K \in \mesh.
\end{align}
\end{subequations}
\end{theorem}

\begin{proof}
The proof of the convergence of $(\rho\exm)_{m \in \xN}$ towards $\rho =1$ when $m$ tends to $+ \infty$ is the same as in the continuous case. It follows from the combination of the lower bounds on $\Pi_\gamma(\rho)$ proven in Lemma \ref{lem:pi} and the global entropy estimate \eqref{eq:corr_stap_bis}.

Since $\rho\exm\to 1$ as $m\to+\infty$, the discrete density unknowns tend to $1$ in every cell. Hence, by the kinetic energy part of \eqref{eq:corr_stap_bis}, there exists a minimum density $\rho_{{\rm min},\disc}>0$ depending on the fixed mesh such that $\norm{\bfu\exm}_{\xL^2(\Omega)^d}\leq \sqrt{\frac{2C}{\rho_{{\rm min},\disc}}}$ for all $m\in\xN$. Moreover, the sequence $(\deltap\exm)_{m \in \xN}$ is also bounded by Lemma \ref{lmm:sch:controle:p}.
By the Bolzano-Weierstrass theorem and a norm equivalence argument, there exists a subsequence of $(\bfu\exm,\deltap\exm)_{m \in \xN}$ which tends, in any discrete norm, to a limit $(\bfu,\deltap)$.
Passing to the limit cell-by-cell in \eqref{eq:correction_scheme}, one obtains that $(\bfu,\deltap)$ is a solution to \eqref{eq:limit_correction_scheme}.
Since this solution is unique, the whole sequence converges, which concludes the proof.
\end{proof}
\begin{remark}[Control of the velocity and extension to the Euler equations]
As for the implicit scheme, one can extend the result of Theorem \ref{Theorem:esch} to the inviscid case where $\mu=\lambda=0$. Indeed, the bound on the sequence of velocities is obtained thanks to the control by the estimate \eqref{eq:corr_stap_bis} of the kinetic energy part of the global entropy, invoking the convergence for the densities towards 1.
 \end{remark}

% --------------------------------------------------------------------------------------------------------------------------
%
\section{Asymptotic analysis of the zero Mach limit for a semi-implicit scheme} \label{sec:proj2}

Another interesting semi-implicit scheme for low viscosity flows (typically a viscosity $ \mu$ which is of the same order of magnitude as the space step $h_\disc$) is a scheme, where in the momentum equation, only the pressure is treated in an implicit way (which is mandatory for stability reasons). 

\medskip
Let $\delta t>0$ be a constant time step. 
The approximate solution $(\rho^n,\bfu^n)\in\xL_\mesh(\Omega)\times\xbfH_{\edges,0}(\Omega)$ at time $t_n=n\delta t$ for $1\leq n\leq N=\Ent{T/\delta t}$ is computed through the following algorithm:

\medskip
\begin{minipage}{0.9\textwidth}
Knowing $(\rho^n,\bfu^n)\in\xL_\mesh(\Omega)\times\xbfH_{\edges,0}(\Omega)$, solve for $\rho^{n+1}\in\xL_\mesh(\Omega)$ and $\bfu^{n+1}\in\xbfH_{\edges,0}(\Omega)$:
\end{minipage}
\begin{subequations}\label{eq:explicit_scheme}
\begin{align} \label{eq:esch_mass} & 
\dfrac 1 {\delta t}(\rho^{n+1}_K-\rho^n_K) + \dive(\rho^{n+1} \bfu^{n+1})_K = 0, &\forall K \in \mesh,
\\[2ex] \label{eq:esch_mom} &
\dfrac 1 {\delta t}\ \bigl(\rho^n_\Ds \bfu^{n+1}_\edge-\rho^{n-1}_\Ds \bfu_\edge^n \bigr)
+ \divv(\rho^n \bfu^n \otimes \bfu^n)_\edge^{\rm up}
- \divv (\bftau(\bfu^n))_\edge
+ \dfrac{1}{\ma^2}\,(\gradi p^{n+1})_\edge
=0, & \forall \edge \in \edgesint.
\end{align}
\end{subequations}

\medskip
Computing $\bfu^{n+1}_\edge$ from the second equation and inserting in the first one yields a nonlinear diffusion-convection-reaction problem on the density $\rho^{n+1}$ or the pressure $p^{n+1}$.
This algorithm may also equivalently be set under the form of a pressure-correction scheme:
\begin{subequations}\label{eq:correction_e_scheme}
\begin{align} \nonumber &
\mbox{{\it Prediction step} -- Compute $\tilde \bfu^{n+1} \in\xbfH_{\edges,0}(\Omega)$ by:}
\\[2ex] \label{eq:e_sch_mom} & \quad
\dfrac 1 {\delta t}\ \bigl(\rho^n_\Ds \tilde \bfu^{n+1}_\edge-\rho^{n-1}_\Ds \bfu_\edge^n \bigr)
+ \divv(\rho^n  \bfu^n \otimes \bfu^n)_\edge^{\rm up}
- \divv (\bftau(\bfu^n))_\edge
=0, & \forall \edge \in \edgesint.
\\[3ex] \nonumber &
\mbox{{\it Correction step} -- Solve for $\rho^{n+1}\in\xL_\mesh(\Omega)$ and $\bfu^{n+1}\in\xbfH_{\edges,0}(\Omega)$:}
\\[1ex] \label{eq:e_sch_cor} & \quad
\dfrac 1 {\delta t}\ \rho^n_\Ds\ (\bfu^{n+1}_\edge-\tilde \bfu_\edge^{n+1})
+ \dfrac{1}{\ma^2}\,(\gradi p^{n+1})_\edge
=0, & \forall \edge \in \edgesint.
\\[2ex] \label{eq:e_sch_mass} & \quad
\dfrac 1 {\delta t}(\rho^{n+1}_K-\rho^n_K) + \dive(\rho^{n+1} \bfu^{n+1})_K = 0, & \forall K \in \mesh.
\end{align}
\end{subequations}

\medskip
The $\xL^2$-stability of this scheme is expected to be ensured under a CFL restriction on the time step of the form $\delta t \leq c(h_\disc/|\bfu|+h_\disc^2/\mu)$, provided that an upwind space discretization be used in the convection term of the momentum balance equation.
Hence, for the semi-implicit scheme \eqref{eq:explicit_scheme}, this convection term is defined as:
\begin{equation} \label{eq:div_conv_up}
\divv(\rho \bfu \otimes \bfv)_\edge^{\rm up} = \sum_{\edged\in\edgesd(D_\edge)} \fluxd(\rho,\bfu)\ \bfv_\edged, \qquad  \forall \edge \in \edgesint,
\end{equation}
where the dual mass fluxes $\fluxd(\rho,\bfu)$ are defined as previously in Section \ref{sec:vel_conv_op}, and where the approximation of the convected velocity on a dual face $\edged=D_\edge|D_\edge'$ is computed with the upwind technique: $\bfv_\edged =\bfv_\edge$ if $\fluxd(\rho,\bfu)\geq 0$, $\bfv_\edged = \bfv_{\edge'}$ otherwise.

\medskip
The initial data $(\rho_0^\ma,\bfu_0^\ma)$ is assumed to satisfy $\rho_0^\ma >0$, $\rho_0^\ma \in \xL^\infty(\Omega)$, $\bfu_0^\ma \in \xH_0^1(\Omega)^d$ and 
\begin{equation} \label{eq:u0_rho0_wp_exp}
\norm{\bfu_0^\ma}_{\xH^1(\Omega)^d}  \ + \frac{1}{\ma}\,\norm{\dive \,\bfu_0^\ma}_{\xL^2(\Omega)} \ + \frac{1}{\ma}\,\norm{\rho^\ma_0- 1}_{\xL^\infty(\Omega)}  \leq  C,
\end{equation}
for a real number $C$ independent of $\ma$. 
Moreover, we suppose that $\bfu_0^\ma$ converges in $\xL^2(\Omega)^d$ towards a function $\bfu_0\in\xL^2(\Omega)^d$.
Note that these hypotheses are less restrictive than assuming well-prepared initial data since the density is here assumed to be close to $1$ with an $\ma$ rate (and not $\ma^2$).

\medskip
The initialization of the scheme \eqref{eq:explicit_scheme} is performed similarly to that of the pressure-correction scheme.
First, $\rho^0$ and $\bfu^0$ are given by the average of the initial conditions $\rho_0^\ma$ and $\bfu_0^\ma$ respectively on the primal cells and on the faces of the primal cells; then we compute $\rho^{-1}$ by solving the mass balance equation \eqref{eq:esch_mass} for $n=-1$, where the unknown is $\rho^{-1}$ and not $\rho^0$.
As in the case of the pressure correction scheme, one may prove that the discrete densities $\rho_K^{-1}$, $K\in\mesh$, are close to $1$ for $\ma$ small enough and therefore positive and bounded.
%
% ----------------------------------------------------
%
\subsection{A priori estimates}

The discrete renormalization identities are again valid for the scheme \eqref{eq:explicit_scheme} , thanks to the fact that the mass balance \eqref{eq:esch_mass} is satisfied. The proof is identical to that of Lemmas \ref{lem:i_renorm} and \ref{lem:sch_cor_remorm}.

\begin{lemma}[Discrete renormalization property] \label{lem:expl_remorm}
A solution to the system \eqref{eq:explicit_scheme} satisfies for all $K\in\mesh$ and $0\leq n \leq N-1$ the two following identities:
\begin{align} & \label{eq:expl_renorm}
\dfrac 1 {\delta t} \Big ( \psi_\gamma(\rho^{n+1}_K)-\psi_\gamma(\rho^n_K) \Big ) + \dive(\psi_\gamma(\rho^{n+1}) \bfu^{n+1})_K 
+  \eos(\rho_K^{n+1}) \, \dive( \bfu^{n+1})_K + R_K^{n+1} = 0, 
\\ & \nonumber
\dfrac 1 {\delta t} \Big ( \Pi_\gamma(\rho^{n+1}_K)-\Pi_\gamma(\rho^n_K) \Big )
+ \dive\Big(\big(\psi_\gamma(\rho^{n+1}) - \psi_\gamma'(1)\,\rho^{n+1}\big)\, \bfu^{n+1} \Big)_K 
\\ & \label{eq:expl_renorm:Pi} \hspace{50ex}
+ \eos(\rho_K^{n+1}) \, \dive( \bfu^{n+1})_K + R_K^{n+1} = 0,
\end{align}
where $R_K^{n+1}$ has the same expression as in Lemma \ref{lem:i_renorm}.
\end{lemma}

\medskip
Let us now turn to the discrete kinetic energy balance which leads to the $\xL^2$ stability of the scheme.
We begin with the following proposition which characterizes the rigidity matrix associated with the discretization of the diffusion term in the momentum equation.

\begin{proposition} \label{defi:rigidity}
Let $\mcal{A}$ be the $d\sharp\edgesint \times d\sharp\edgesint$ rigidity matrix associated with the finite element discretization of the diffusion term.  The matrix $\mcal{A}$ is the block matrix $\mcal{A}=(\mcal{A}_{\edge,\edge'})_{\edge,\edge'\in\edgesint}$, where for $\edge,\edge'\in\edgesint$, $\mcal{A}_{\edge,\edge'}$ is the $d\times d$ matrix defined by:
\begin{equation}\label{eq:exp_A}
\mcal{A}_{\edge,\edge'} = \sum_{K\in\mesh} \Big ( \mu  \int_K \gradi \zeta_\edge \cdot \gradi \zeta_{\edge '} \bfI
+ (\mu+\lambda) \int_K \gradi \zeta_\edge \otimes \gradi \zeta_{\edge'} \Big ),
\end{equation}
and $\mcal{A}$ is a symmetric positive-definite matrix.
\end{proposition}

\begin{proof}
Let us identify $\bfu=(\bfu_\edge)_{\edge\in\edgesint}$ and $\bfv=(\bfv_\edge)_{\edge\in\edgesint}$ with vectors in $\xR^{d\sharp\edgesint}$ and denote $(.,.)$ the canonical scalar product in  $\xR^{d\sharp\edgesint}$. The result follows from the elementary computation:
\[ \begin{aligned}
(\mcal{A}\bfu,\bfv) 
& = \sum_{\edge\in\edgesint} \sum_{\edge'\in\edgesint} \mcal{A}_{\edge,\edge'} \bfu_{\edge'} \cdot \bfv_\edge \\ 
& = \sum_{\edge \in \edgesint} |D_\edge|\ \big( - \divv (\bftau(\bfu))_\edge \big ) \cdot \bfv_\edge \\
& = \sum_{K\in\mesh}   \mu  \int_K  \gradi \hat \bfu : \gradi \hat \bfv  + (\mu+\lambda) \int_K (\dive\, \hat\bfu)(\dive\, \hat\bfv).
\end{aligned} \]
The expression \eqref{eq:exp_A} of the block-entries of the matrix $\mcal{A}$ then follows by developping the expression of $\hat\bfu$ and $\hat\bfv$ on the basis of the shape functions and identifying the terms.
\end{proof}

The following lemma states the $\xL^2$-stability of an explicit upwind scheme for the convection-diffusion equation \eqref{eq:e_sch_mom}, under some CFL restriction on the time step.

\begin{lemma}[$\xL^2$-stability of the prediction step] \label{lmm:expl_stab_pred}
Any solution to \eqref{eq:e_sch_mom}, and thus to the explicit scheme \eqref{eq:explicit_scheme}, satisfies the following inequality, for all  $0\leq n \leq N-1$:
\begin{equation} \label{eq:expl_pred}
\dfrac 1 {2\delta t} \sum_{\edge\in\edgesint} |D_\edge|\,\Big ( \rho_\Ds^n\,|\tilde \bfu_\edge^{n+1}|^2 - \rho_\Ds^{n-1}\,|\bfu_\edge^n|^2 \Big) 
+  R_{\edges}^{n+1}\leq 0,
\end{equation}
where, denoting $\varrho(\mcal{A})$ the spectral radius of the matrix $\mcal{A}$, the remainder term  $R_{\edges}^{n+1}$ is given by:
\begin{equation} \label{remainder:CFL}
R_{\edges}^{n+1}= \sum_{\edge\in\edgesint} \Big ( 
\frac{\rho_\Ds^n |D_\edge|}{2 \delta t} 
- \dfrac 12 \sum_{\edged\in\edgesd(D_\edge)}\,\bigl( \fluxd(\rho^n,\bfu^n)\bigr )^{-}
- \dfrac 14 \varrho(\mcal{A})
\Big )\,|\tilde \bfu_\edge^{n+1}-\bfu_\edge^n|^2.
\end{equation}
\end{lemma}

\begin{proof}
Taking the scalar product of the momentum equation \eqref{eq:e_sch_mom} with $|D_\edge|\, \tilde \bfu_\edge^{n+1}$, we obtain $T^{\rm conv}_\edge + T^{\rm diff}_\edge=0$, with:
\[
\begin{aligned} &
T^{\rm conv}_\edge = \Bigl ( \dfrac {|D_\edge|}{\delta t}\ \bigl(\rho^n_\Ds \tilde \bfu^{n+1}_\edge-\rho^{n-1}_\Ds \bfu_\edge^n \bigr)
+ \sum_{\edged\in\edges(D_\edge)} \fluxd(\rho^n,\bfu^n)\,\bfu_\edged^n \Bigr ) \cdot \tilde \bfu_\edge^{n+1},
\\[1ex] &
T^{\rm diff}_\edge = - |D_\edge| \, \divv (\bftau(\bfu^n))_\edge \cdot \tilde \bfu_\edge^{n+1}.
\end{aligned}
\]
For the convection term, the dual density unknowns and mass fluxes are chosen so as to have:
\[
\frac{|D_\edge|}{\delta t} \ (\rho^n_\Ds-\rho^{n-1}_\Ds)+ \sum_{\edged\in\edges(D_\edge)} \fluxd(\rho^n,\bfu^n) = 0, \qquad \forall \edge\in\edgesint,
\]
which necessitates a time-shift in order to exploit the mass balance at the previous time step.
Hence, following the proof of \cite[Lemma A.2]{her-17-con}, we obtain:
\begin{multline*}
T^{\rm conv}_\edge 
= \dfrac {|D_\edge|}{2\delta t} \Big ( \rho_\Ds^n\,|\tilde \bfu_\edge^{n+1}|^2 - \rho_\Ds^{n-1}\,|\bfu_\edge^n|^2\Big ) 
+ \frac 12 \sum_{\edged\in\edges(D_\edge)} \fluxd(\rho^n,\bfu^n) |\bfu_\edged^n|^2 
+ \dfrac {|D_\edge|}{2\delta t} \rho_\Ds^n\,|\tilde \bfu_\edge^{n+1}-\bfu_\edge^n|^2
\\
- \frac 12 \sum_{\edged\in\edges(D_\edge)} \fluxd(\rho^n,\bfu^n) |\bfu_\edged^n-\bfu_\edge^n|^2 
+  \sum_{\edged\in\edges(D_\edge)} \fluxd(\rho^n,\bfu^n) (\bfu_\edged^n-\bfu_\edge^n) \cdot (\tilde \bfu_\edge^{n+1}-\bfu_\edge^n).
\end{multline*}
Thanks to the upwind choice for $\bfu_\edged^n$, the term $\bfu_\edged^n-\bfu_\edge^n$ vanishes whenever the dual flux $\fluxd(\rho^n,\bfu^n)$ is non-negative. Applying Young's inequality to the product in the last term yields:
\begin{multline}
\label{ineq:Tconv}
T^{\rm conv}_\edge 
\geq \dfrac {|D_\edge|}{2\delta t} \Big ( \rho_\Ds^n\,|\tilde \bfu_\edge^{n+1}|^2 - \rho_\Ds^{n-1}\,|\bfu_\edge^n|^2\Big ) 
+ \frac 12 \sum_{\edged\in\edges(D_\edge)} \fluxd(\rho^n,\bfu^n) |\bfu_\edged^n|^2 
+ \dfrac {|D_\edge|}{2\delta t} \rho_\Ds^n\,|\tilde \bfu_\edge^{n+1}-\bfu_\edge^n|^2\\
- \frac 12 \sum_{\edged\in\edges(D_\edge)} \fluxd(\rho^n,\bfu^n)^- |\tilde \bfu_\edge^{n+1}-\bfu_\edge^n|^2.
\end{multline}
For the diffusion term, we observe that:
\[
 \sum_{\edge\in\edgesint} T^{\rm diff}_\edge = (\mcal{A}\bfu^n,\tilde \bfu^{n+1}) = (\mcal{A}\bfu^n,\tilde \bfu^{n+1}-\bfu^n) + (\mcal{A}\bfu^n,\bfu^n).
\]
The Cauchy-Schwarz inequality for the scalar product associated with the real positive symmetric matrix $\mcal{A}$ yields
\[
 (\mcal{A}\bfu^n,\tilde \bfu^{n+1}-\bfu^n) \geq
- (\mcal{A}\bfu^n,\bfu^n)^{\frac 12} \ (\mcal{A}(\tilde \bfu^{n+1}-\bfu^n),\tilde \bfu^{n+1}-\bfu^n)^{\frac 12}.
\]
Applying Young's inequality yields:
\begin{equation} \label{ineq:Tdiff}
\sum_{\edge\in\edgesint} T^{\rm diff}_\edge 
\, \geq \, - \frac 14 (\mcal{A}(\tilde \bfu^{n+1}-\bfu^n),\tilde\bfu^{n+1}-\bfu^n)
\, \geq \, - \frac 14 \varrho(\mcal{A}) \, \sum_{\edge\in\edgesint} |\tilde \bfu_\edge^{n+1}-\bfu_\edge^n|^2.
\end{equation}
By the conservativity of the dual fluxes, taking the sum of \eqref{ineq:Tconv} over $\edge\in\edgesint$ and summing the result with \eqref{ineq:Tdiff} yields the expected inequality \eqref{eq:expl_pred}.
\end{proof}

We are now in position to state the following local-in-time kinetic energy balance.

\begin{lemma}[Discrete kinetic energy balance] \label{lmm:expl_ke}
Any solution to the explicit scheme \eqref{eq:explicit_scheme} satisfies the following inequality, for all  $0\leq n \leq N-1$:
\begin{multline} \label{eq:expl_ke}
\dfrac 1 {2\delta t} \sum_{\edge\in\edgesint} |D_\edge|\,\Big ( \rho_\Ds^n\,|\bfu_\edge^{n+1}|^2 - \rho_\Ds^{n-1}\,|\bfu_\edge^n|^2 \Big) 
+ \frac{1}{\ma^2} \sum_{\edge\in\edgesint} |D_\edge|\,(\gradi p^{n+1})_\edge \cdot \bfu_\edge^{n+1}
\\
+ \frac{\delta t}{2\,\ma^4} \sum_{\edge\in \edgesint}  \frac{|D_\edge|}{\rho^n_\Ds}\,\bigl|(\gradi p^{n+1})_\edge \bigr|^2
+ R_{\edges}^{n+1}
\leq 0,
\end{multline}
where $R_{\edges}^{n+1}$ is defined in Lemma \ref{lmm:expl_stab_pred}.
\end{lemma}

\begin{proof}
As for the first pressure correction scheme, we write the velocity correction equation as:
\[
\Bigl(\frac{\rho^n_\Ds}{\delta t} \Bigr)^{1/2}\, \bfu_\edge^{n+1}
+ \Bigl( \frac{\delta t }{\rho^n_\Ds}\Bigr)^{1/2}\, \frac{1}{\ma^2}\, (\gradi p^{n+1})_\edge
= \Bigl(\frac{\rho^n_\Ds}{\delta t}  \Bigr)^{1/2}\,\tilde \bfu_\edge^{n+1},
\]
square this relation and sum it with \eqref{eq:expl_pred}, which yields the desired inequality.
\end{proof}

We may now state the following result.

\medskip
\begin{lemma}[Local-in-time discrete entropy inequality, existence of a solution]
\label{lmm:loc_ener_esch}
Let $\ma>0$ and assume that the initial data satisfies \eqref{eq:u0_rho0_wp_exp}.
Then, for $\ma$ small enough to ensure that $\rho^{-1}$ is positive, there exists a solution $(\rho^n,\bfu^n)_{0\leq n\leq N}$ to the scheme \eqref{eq:explicit_scheme}, and for $1\leq n\leq N$, the density $\rho^n$ is positive.
Moreover, assuming that the time-step satisfies the following CFL restriction:
\begin{equation} \label{CFL}
\delta t \leq \min_{\edge\in\edgesint} \, \frac{4\,\rho_\Ds^n |D_\edge|}
{\displaystyle 2 \sum_{\edged\in\edgesd(D_\edge)}\,\bigl( \fluxd(\rho^n,\bfu^n)\bigr )^{-} + \varrho(\mcal{A})},
\end{equation}
the following inequality holds:
\begin{multline} \label{esch_loc_ener}
\frac 1 2 \sum_{\edge\in \edgesint} |D_\edge| \Big ( \rho^n_\Ds\ |\bfu^{n+1}_\edge|^2 -  \rho^{n-1}_\Ds\ |\bfu^n_\edge|^2 \Big )
+ \frac 1 {\ma^2} \sum_{K\in \mesh} |K| \Big ( \Pi_\gamma(\rho^{n+1}_K) - \Pi_\gamma(\rho^n_K) \Big )
\\
+ \frac{\delta t^2}{2\,\ma^4} \sum_{\edge\in \edgesint}  \frac{|D_\edge|}{\rho^n_\Ds}\,\bigl|(\gradi p^{n+1})_\edge \bigr|^2
+ \mathcal{R}^{n+1} \leq 0,
\end{multline}
where $\mathcal{R}^{n+1} = R^{n+1}_{\edges} + \ma^{-2}\sum_{K \in \mesh} R_K^{n+1}\geq 0$.
\end{lemma}

\begin{proof}
The positivity of the density is a consequence of the properties of the upwind choice \eqref{eq:def_FKedge} for $\rho$ in the mass balance of the correction step \eqref{eq:esch_mass}.
After multiplication by $\ma^{-2}|K|$, we sum the renormalization identity \eqref{eq:expl_renorm:Pi} over the primal cells, and sum the obtained relation with the kinetic energy balance \eqref{eq:expl_ke}.
Since the discrete gradient and divergence operators are dual with respect to the $\xL^2$ inner product (see \eqref{eq:grad-div}), we get \eqref{esch_loc_ener}. Under the CFL condition \eqref{CFL}, the remainder term $R^{n+1}_{\edges}$ defined in \eqref{remainder:CFL} is non-negative.
The existence of a solution $(\rho^{n+1},\bfu^{n+1})$ to the scheme \eqref{eq:explicit_scheme} follows form the Brouwer fixed point theorem, by an easy adaptation of the proof of \cite[Proposition 5.2]{eym-10-conv}.
\end{proof}

The restriction \eqref{CFL} on the time-step is a convective-diffusive CFL condition. In particular, it is satisfied if the time step simultaneously satisfies the following two conditions:
\[
\delta t \leq \min_{\edge\in\edgesint} \, \frac{\rho_\Ds^n |D_\edge|}{\displaystyle  \sum_{\edged\in\edgesd(D_\edge)}\,\bigl( \fluxd(\rho^n,\bfu^n)\bigr )^{-}}
\quad \text{and} \quad \delta t \leq \min_{\edge\in\edgesint} \, \frac{2\,\rho_\Ds^n |D_\edge|}{ \varrho(\mcal{A})}.
\]
The first condition is a convective CFL restriction associated with the velocity of the fluid, which is consistent with the explicit upwind discretization of the momentum convection term.
As for the second condition, it is also a classical diffusive CFL restriction associate with an explicit treatment of the diffusion term.
Indeed, if the Lam\'e coefficients $\mu$ and $\lambda$ are of the same order of magnitude, usual regularity assumptions on the mesh imply that: 
\[
\varrho(\mcal{A}) \propto \mu\, h_\disc^{d-2},
\]
where $h_\disc$ is a characteristic length of the mesh cells.
Therefore, if the viscosity $\mu$ is small, typically $\mu\approx h_\disc$, the diffusive CFL restriction is comparable to the convective CFL restriction.

\medskip
\begin{lemma}[Global discrete entropy inequality]
Let $\ma>0$ and assume that the initial data $(\rho_0^\ma,\bfu_0^\ma)$ satisfies \eqref{eq:u0_rho0_wp_exp}.
By Lemma \ref{lmm:loc_ener_esch}, for $\ma$ small enough, there exists a solution $(\rho^n,\bfu^n)_{0\leq n\leq N}$ to the scheme \eqref{eq:explicit_scheme}. Moreover, if the time-step satisfies the CFL restriction \eqref{CFL} for all $1\leq n\leq N$ , then there exists $C>0$ independent of $\ma$ such that, for $\ma$ small enough and for all $1\leq n\leq N$ :
\begin{equation}
\label{eq:esch_stab}
\frac 1 2  \sum_{\edge\in \edgesint} |D_\edge| \ \rho^{n-1}_\Ds\ |\bfu^n_\edge|^2
+ \frac 1 {\ma^2} \sum_{K\in \mesh}|K| \, \Pi_\gamma(\rho^n_K)
+ \frac{\delta t^2}{2\,\ma^4} \sum_{k=1}^n \sum_{\edge\in \edgesint}  \frac{|D_\edge|}{\rho^{k-1}_\Ds}\,\bigl|(\gradi p^k)_\edge \bigr|^2
\leq C.
\end{equation}
\end{lemma}

\begin{proof}
Multiplying equation \eqref{esch_loc_ener} by $\delta t$ and summing over the time steps yields for $1\leq n\leq N$:
\begin{multline} \label{eq:expl_stab}
\frac 1 2  \sum_{\edge\in \edgesint} |D_\edge| \ \rho^{n-1}_\Ds\ |\bfu^n_\edge|^2
+ \frac 1 {\ma^2} \sum_{K\in \mesh}|K| \, \Pi_\gamma(\rho^n_K)
+ \frac{\delta t^2}{2\,\ma^4} \sum_{k=1}^n \sum_{\edge\in \edgesint}  \frac{|D_\edge|}{\rho^{k-1}_\Ds}\,\bigl|(\gradi p^k)_\edge \bigr|^2
+ \mathcal{R}^n
\\
\leq \ 
\frac 1 2  \sum_{\edge\in \edgesint} |D_\edge| \ \rho^{-1}_\Ds\ |\bfu^0_\edge|^2
+\frac 1 {\ma^2} \sum_{K\in \mesh}|K| \ \Pi_\gamma(\rho^0_K),
\end{multline}
with $\displaystyle \mathcal{R}^n= \sum_{k=0}^{n-1} \big (R^{k+1}_{\edges} + \ma^{-2}\sum_{K \in \mesh} R_K^{k+1} \big ) \geq 0$.

Let us prove that the right hand side of \eqref{eq:expl_stab} is uniformly bounded for all $\ma$ small enough.
As in Lemma \ref{lmm:sch:init}, one may prove that under \eqref{eq:u0_rho0_wp_exp}, $\rho_K^{-1}$ is bounded for all $K\in\mesh$ for $\ma$ small enough and therefore so is $\rho_\Ds^{-1}$ for all $\edge\in\edgesint$. 
Hence, since $\bfu_0^\ma$ is uniformly bounded in $\xH^1(\Omega)^d$ by \eqref{eq:u0_rho0_wp_exp}, a classical trace inequality yields the boundedness of the first term. By \eqref{eq:u0_rho0_wp_exp}, one has $|\rho_K^0-1|\leq C\ma$ for all $K\in\mesh$. 
Hence, by \eqref{upperPi}, the second term is also uniformly bounded with respect to $\ma$.
\end{proof}

From now on, we need to assume that we use a time step independent of the Mach number and satisfying the constraint \eqref{CFL} (which involves the velocity field, which itself depends on $\ma$).
The existence of such a time step is proven in Appendix \ref{sec:tstep_cfl}.

\medskip
Under this assumption, by the same arguments as for the pressure correction scheme, we get that the dynamic pressure is controlled independently of $\ma$; this is stated in the following lemma.

\begin{lemma}[Control of the pressure]
\label{lmm:esch:controle:p}
Let $\ma>0$ and assume that the initial data $(\rho_0^\ma,\bfu_0^\ma)$ satisfies \eqref{eq:u0_rho0_wp_exp}.
Then, there exists a solution $(\rho^n,\bfu^n)_{0\leq n\leq N}$ to the scheme \eqref{eq:explicit_scheme}.
Let  $p^n=\eos(\rho^n)$ and define $\deltap^n = \lbrace \deltap_K^n,\,K\in\mesh \rbrace $ where $\deltap_K^n = (p_K^n-m(p^n))/\ma^2$ with $m(p^n)$ the mean value of $p^n$ over $\Omega$.
If the time step $\delta t$ satisfies the CFL condition \eqref{CFL} for all $1\leq n\leq N$ independently of the Mach number $\ma$, then, one has, for all $1\leq n \leq N$:
\begin{equation} \label{eq:esch:controle:p}
\norm{\deltap^n} \leq C_{\disc,\delta t},
\end{equation}
where  the real number $C_{\disc,\delta t}$ depends on the mesh and the time step but not on $\ma$, and $\norm{\cdot}$ stands for any norm on the space of discrete functions.
\end{lemma}
%
% ------------------------------------------
%
\subsection{Incompressible limit of the scheme}

By the same proof as for the previous pressure correction scheme, we have the following convergence result. An analogous result is also valid in the inviscid case $\mu=\lambda=0$.

\begin{theorem}[Incompressible limit of the semi-implicit scheme \eqref{eq:explicit_scheme}]\ \\
Let $(\ma\exm)_{m\in \xN}$ be a sequence of positive real numbers tending to zero.
Let the associated sequence of initial data $(\rho_0^{\ma\exm},\bfu_0^{\ma\exm})$ satisfy \eqref{eq:u0_rho0_wp_exp}, and the time step satisfy the CFL condition \eqref{CFL} for all $1\leq n\leq N$ independently of $m$.
Let $(\rho\exm,\bfu\exm)$ be a corresponding sequence of solutions of the scheme \eqref{eq:explicit_scheme}.
Then the sequence $(\rho\exm)_{m \in \xN}$ tends to the constant function $\rho =1$ when $m$ tends to $+ \infty$ in $\xL^\infty((0,T),\xL^\gamma(\Omega))$. Moreover, for all $q\in[1,\min(2,\gamma)]$, there exists $C>0$ such that:
\[
   \norm{\rho\exm-1}_{\xL^\infty((0,T);\xL^q(\Omega))} \leq C\ma\exm, \qquad \text{for $m$ large enough}.
\]

In addition, the sequence $(\bfu\exm, \delta p\exm)_{m \in \xN}$ tends, in any discrete norm, to a limit $(\bfu,\deltap)$ which is the solution to the following inf-sup stable semi-implicit scheme for the incompressible Navier-Stokes equations: 

\medskip
Knowing  $\deltap^n\in\xL_\mesh(\Omega)$ and $\bfu^n\in\xbfH_{\edges,0}(\Omega)$, solve for $\deltap^{n+1}\in\xL_\mesh(\Omega)$ and $\bfu^{n+1}\in\xbfH_{\edges,0}(\Omega)$ :
\begin{subequations}\label{eq:limit_explicit_scheme}
\begin{align} \label{eq:limit_esch_mass} & 
\dive(\bfu^{n+1})_K = 0, &\forall K \in \mesh,
\\[2ex]  \label{eq:limit_esch_mom} &
\dfrac 1 {\delta t}\ \bigl(\bfu^{n+1}_\edge-\bfu_\edge^n \bigr)
+ \divv(\bfu^n \otimes \bfu^n)_\edge^{\rm up}
- \divv (\bftau(\bfu^n))_\edge
+ (\gradi \deltap^{n+1})_\edge
=0, & \forall \edge \in \edgesint.
\end{align}
\end{subequations}
\end{theorem}
%
% -------------------------------------------------------------------------------------------------------------------
%
\section{Numerical tests} \label{sec:ana_sol}

We assess the convergence of the scheme on a test case built for this purpose.
Analytical solutions of the 2D barotropic Euler and Navier-Stokes equations are obtained through the following steps: we first derive a compactly supported $H^2(\xR^2)$ solution of the stationary barotropic Euler equations consisting in a standing vortex; then, we obtain a time-dependent solution of the Euler equations by adding a constant velocity motion; finally, we pass to the Navier-Stokes equations by compensating the viscous forces at the right-hand side.
The velocity field of the standing vortex is sought under the form:
\[
\hat \bfu = f(\xi) \begin{bmatrix} -x_2 \\ x_1 \end{bmatrix},\quad \mbox{with } \xi = x_1^2 + x_2^2.
\]
A simple derivation of this expression yields:
\[
(\hat \bfu \cdot \gradi) \hat \bfu = -f(\xi)^2 \begin{bmatrix} x_1 \\ x_2 \end{bmatrix}.
\]
Let us seek the pressure under the form $\hat p=g(\xi)$, so:
\[
\gradi \hat p = 2\, g'(\xi) \begin{bmatrix} x_1 \\ x_2 \end{bmatrix}.
\]
Assuming a perfect gas pressure law, the density is given by $\hat \rho = g(\xi)^{1/\gamma}$.
By construction, these functions satisfy $\dive(\hat\rho\hat\bfu)=0$. Moreover, one has $\hat \rho (\hat \bfu \cdot \gradi) \hat \bfu +\gradi  \hat p = 0 $ if, and only if $-g(\xi)^{1/\gamma}\,f(\xi)^2 + 2\, g'(\xi) =0$ for all $\xi$.
We thus obtain a solution of the stationary Euler equation if $g$ takes the following expression:
\[
g=\Bigl(\frac{\gamma-1}{2\gamma} (F +c_M) \Bigr)^{\gamma/(\gamma-1)},
\]
where $F$ is such that $F'=f^2$ and $c_M$ is a positive real number.
For the present numerical study, we choose $\gamma=3$ and $f(\xi)=10\,\xi^2 (1-\xi)^2$ if $\xi \in (0,1)$, $f=0$ otherwise, which indeed yields an $H^2(\xR^2)$ velocity field.
The associated expression of $F$ is:
\[
F(\xi)=100\ \bigl(\frac 1 5 \xi^5 -\frac 2 3 \xi^6 + \frac 6 7 \xi^7 - \frac 1 2 \xi^8 +\frac 1 9 \xi^9\bigr) \mbox{ if } \xi \in (0,1),
\ F(\xi) = F(1) = \frac{10}{63} \mbox{ otherwise}.
\]
The problem is made unstationary by a time translation: given a constant vector field $\bfa$, the density $\rho$ and the velocity $\bfu$  are deduced from the steady state solution $\hat \rho$ and $\hat \bfu$ by $\rho(\bfx,t)=\hat \rho(\bfx-\bfa t)$ and $\bfu(\bfx,t)=\hat \bfu(\bfx-\bfa t)+\bfa$.
The center of the vortex is initially located at $\bfx_0=(0,0)^t$, the translation velocity $\bfa$ is set to $\bfa=(1,1)^t$, the computational domain is $\Omega=(-1.2,\,2.8)^2$ and the computation is run on the time interval $(0,0.8)$.
We perform several computations keeping the velocity constant (and therefore of order 1, according to the expression of $f$) and varying the Mach number by changing the constant $c_M$ and therefore also the pressure level and the speed of sound (given by $c^2=\gamma\, p ^{(\gamma-1)/\gamma}$).
Choosing $1$ as the reference value for the velocity and the speed of sound outside the vortex as the reference speed of sound, $c_M$ and the Mach number are connected as given in Table \ref{table-mach}:

\medskip
\begin{table}[hbt]
\begin{center}
\begin{tabular}{|c|c|c|c|c|c|}\hline
\rule[-1.2ex]{0pt}{3.8ex} $c_M$ & $1$ & $10^2$ & $10^4$ & $10^6$ & $10^8$
\\ \hline
\rule[-1.2ex]{0pt}{3.8ex} c & $1.08$ & $10.$ & $100$ & $1000$ & $10000$
\\ \hline
\rule[-1.2ex]{0pt}{3.8ex} Ma & $\simeq 1$ & 0.1 & 0.01 & 0.001 & 0.0001
\\ \hline
\end{tabular}
\end{center}
\label{table-mach}
\caption{Values of the speed of sound and Mach number with respect to the constant $c_M$.}	
\end{table}

\medskip
Computations are run with the open-source CALIF$^3$S software developed at IRSN \cite{califs}, with the pressure correction algorithm described in Section \ref{sec:proj1}.
The mesh is a $500\times 500$ uniform grid, and the time step is set at the same value as the space step, \ie\ $\delta t = 0.008$, for all the computations, which corresponds to a CFL number with respect to the material velocity close to $1.5$.

\medskip
\paragraph{Euler equations} -- In the Euler case, since the convection term in the momentum balance is approximated with a centered discretization, the computations are stabilized by taking into account an artificial viscosity given by:
\[
\mu_a = \rho_{ext} \, v_{max}\, h/10,
\]
where $\rho_{ext}$ stands for the density outside the vortex (which depends on the computation), $v_{max} =1.4$ is an approximation of the maximal value of the components of the velocity and $h$ is the space step.
This viscosity is in the order of a fifth of the upwinding-induced numerical viscosity.
We plot on Figure \ref{fig:vy_euler} the second component of the velocity obtained at $t=0.8$ along the line $x_2=0.8$, which crosses the center of the vortex.
The results are almost independent of the Mach number (in fact, the curves are superimposed on the figure).
To check the differences, a zoom of the curves near the minimum value of the velocity is shown on Figure \ref{fig:vy_zoom_euler}; the maximum of the differences is close to $0.0016$, while the amplitude of the analytical velocity variation is equal to $1$.

\begin{figure}[!t]
\begin{center} \includegraphics[width=0.8\textwidth]{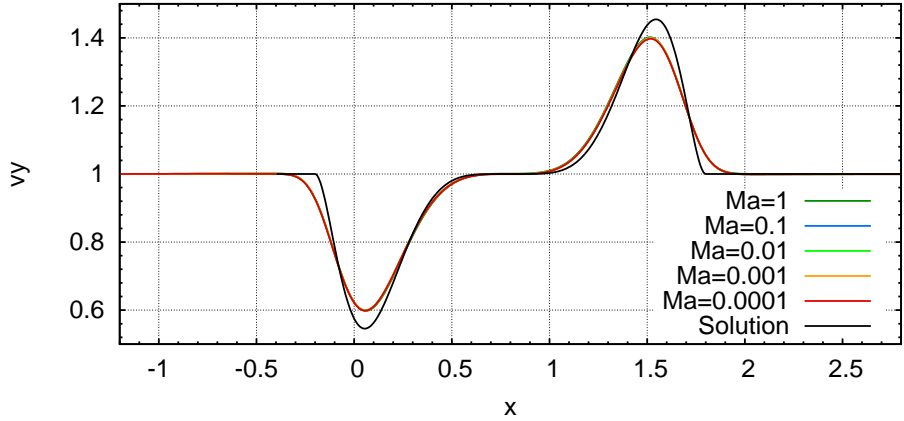} \end{center}
\caption{\label{fig:vy_euler}
Euler case -- Second component of the velocity at $t=0.8$ along the line $x_2=0.8$ for various Mach numbers (all the curves canot be distinguished) and analytical solution.}
\end{figure}

\begin{figure}[!t]
\begin{center} \includegraphics[width=0.8\textwidth]{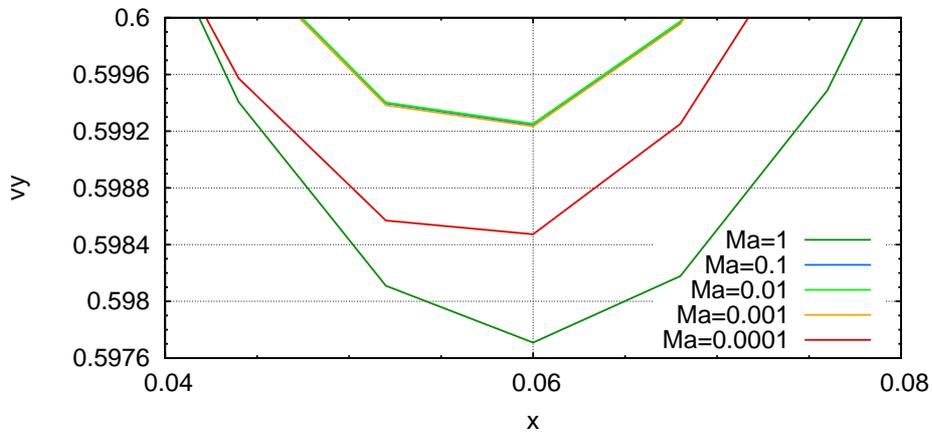} \end{center}
\caption{\label{fig:vy_zoom_euler}
Euler case -- Second component of the velocity at $t=0.8$ along the line $x_2=0.8$ for various Mach numbers ; zoom near the minimum of the velocity.}
\end{figure}

In addition, the expression of the pressure shows that $F$ is much lower than $c_M$, so a Taylor development shows that the quantity $\delta p$ defined by
\[
\delta p =(p-p_{ext})/c,
\]
with $p_{ext}$ the pressure outside the vortex and $c$ the speed of sound given in the above table, should be approximatively independent of the Mach number.
This quantity is plotted on Figure \ref{fig:p_euler}, which shows that it is indeed the case (the observed discrepancy when the Mach number is close to $1$ may be attributed to the fact that $F/c_M$ takes in this case its greatest value, and the Taylor development is less accurate).

\medskip
Finally, the $L^1$ norm of the difference between the numerical velocity and the piecewise constant function obtained by taking, on each diamond cell, the value of the continuous solution at the cell mass center is, at $t=0.8$: $0.192$ for a Mach number $Ma$ close to $1$, $0.189$ for $Ma=0.1$ and $0.187$ for the other values of the Mach number.
For the pressure, the same discrete $L^1$ norm of the difference between the numerical and analytical solutions scales as the magnitude of the pressure variations, which, in turn, as said before, scales as the speed of sound $c$.
The ratio between this norm and $c$ reads: $0.0168$, $0.0223$, $0.0229$, $0.0227$ and $0.0264$, for the tested Mach numbers from $1$ to $0.0001$.
The value slightly greater obtained for $Ma=0.0001$ is probably due to the fact that the pressure is so high in this case that the algebraic solvers become less accurate.

\begin{figure}[!t]
\begin{center} \includegraphics[width=0.8\textwidth]{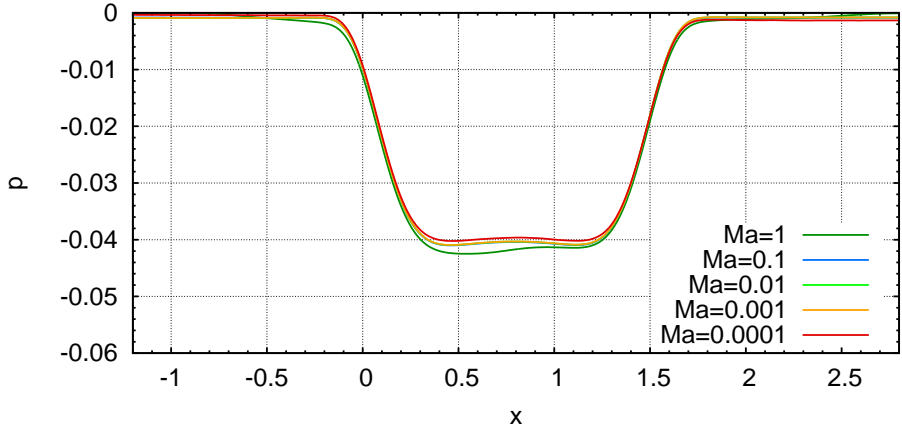} \end{center}
\caption{\label{fig:p_euler}
Euler case -- Difference between the local pressure and the pressure outside the vortex scaled by the speed of sound, along the line $x_2=0.8$, for various Mach numbers.}
\end{figure}

\medskip
\paragraph{Navier-Stokes equations} -- The same analytical solution is used for the Navier-Stokes equations, with a viscosity now given by:
\[
\mu = \rho_{ext}/50,
\]
so a Reynolds number in the range of $50$ (equal to $50$ if the characteristic velocity range is set to $1$, equal to $75$ if its is set to $v_{max}=1.5$).
The corresponding viscous term is compensated by a source term at the right-hand side of the momentum balance equation.
The same curves as for the Euler case are shown on Figures \ref{fig:vy_NS}-\ref{fig:p_NS}.
The numerical error for the velocity (as defined before) is almost independent of the Mach number: it always falls in the interval $(0.099,0.1)$.
For the pressure, the same scaled quantity as before reads: $0.0108$, $0.0138$, $0.0147$, $0.0146$ $0.0183$, for the tested Mach numbers from $1$ to $0.0001$.
The conclusions are thus the same, up to the minor difference that the results in the Navier-Stokes case are slightly more accurate.

\begin{figure}[!t]
\begin{center} \includegraphics[width=0.8\textwidth]{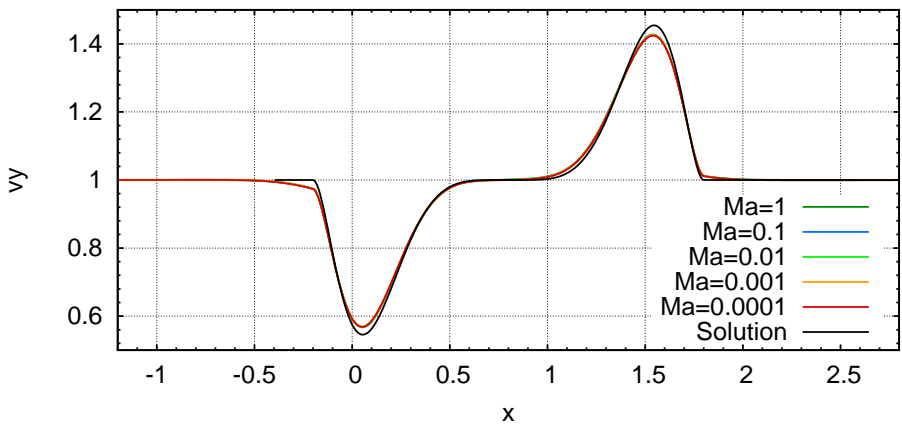} \end{center}
\caption{\label{fig:vy_NS}
Navier-Stokes case -- Second component of the velocity at $t=0.8$ along the line $x_2=0.8$ for various Mach numbers (all the curves canot be distinguished) and analytical solution.}
\end{figure}

\begin{figure}[!t]
\begin{center} \includegraphics[width=0.8\textwidth]{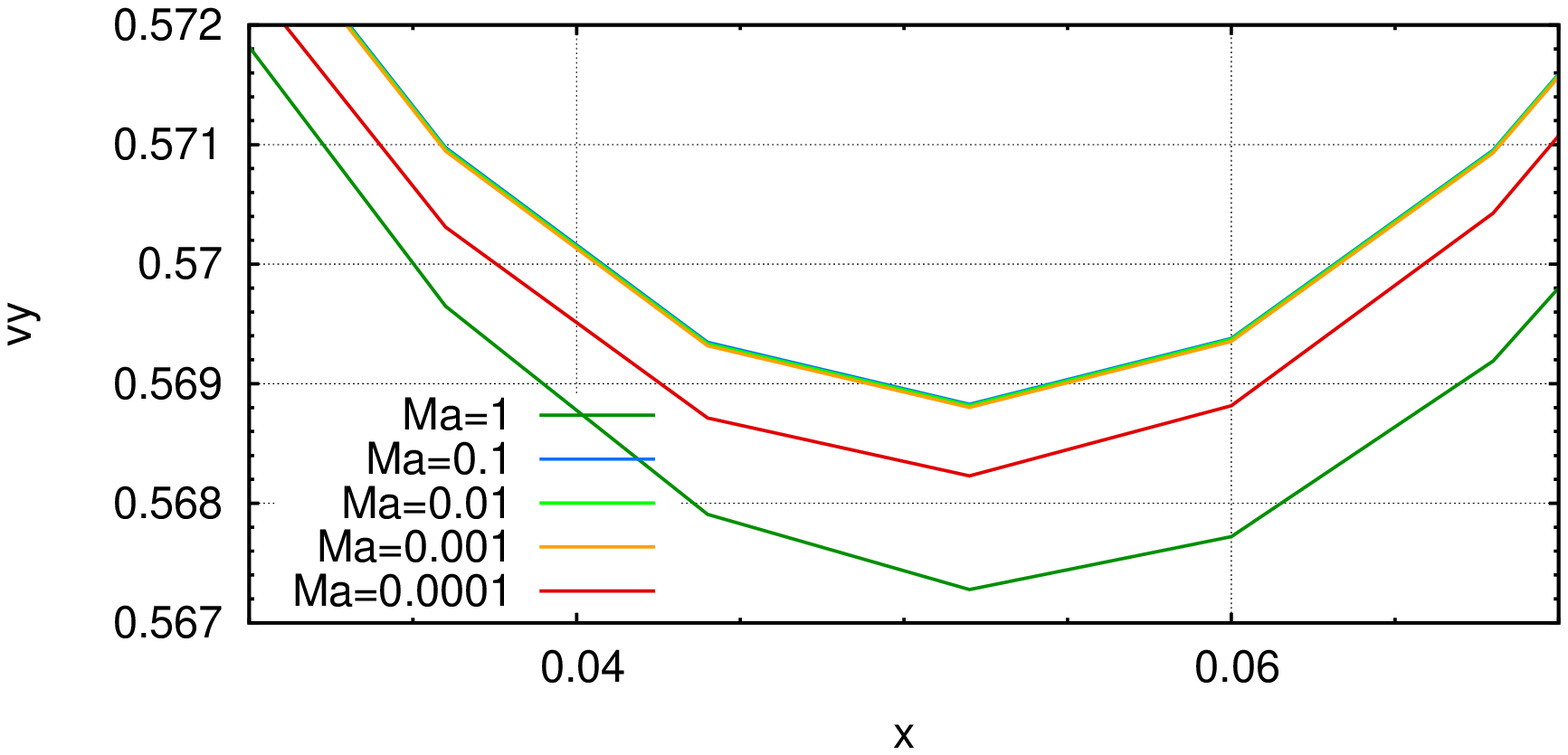} \end{center}
\caption{\label{fig:vy_zoom_NS}
Navier-Stokes case -- Second component of the velocity at $t=0.8$ along the line $x_2=0.8$ for various Mach numbers ; zoom near the minimum of the velocity.}
\end{figure}

\begin{figure}[!t]
\begin{center} \includegraphics[width=0.8\textwidth]{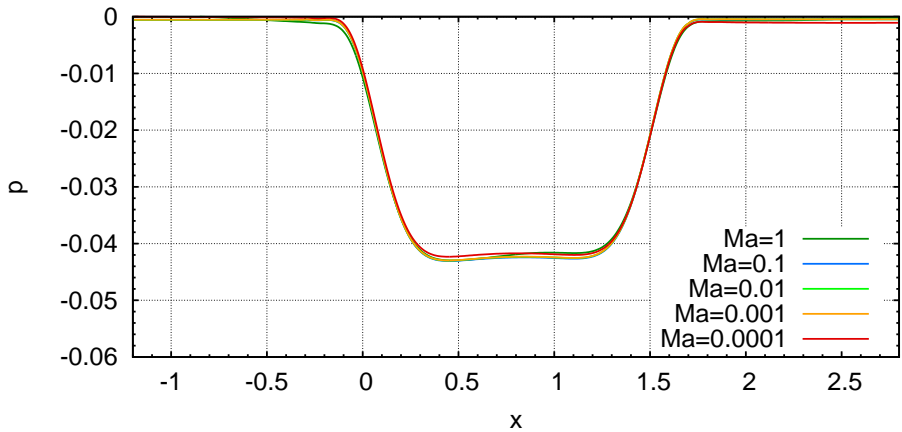} \end{center}
\caption{\label{fig:p_NS}
Navier-Stokes case -- Difference between the local pressure and the pressure outside the vortex scaled by the speed of sound, along the line $x_2=0.8$, for various Mach numbers.}
\end{figure}
%
% -------------------------------------------------------------------------------------------------------------------
%
\appendix
\section{Existence of a time step satisfying the CFL condition \eqref{CFL} for all Mach numbers} \label{sec:tstep_cfl}

In this appendix, we prove that for the semi-implicit scheme \eqref{eq:explicit_scheme}, it is possible to define a time-step $\delta t$ which satisfies the CFL condition \eqref{CFL} independently of the Mach number $\ma$.
This allows to prove the discrete global entropy estimate \eqref{eq:esch_stab} independently of the Mach number.
More precisely, we prove that under a convective CFL condition which is more restrictive than expected ($\delta t\leq C h_\disc^{1+\frac d2}$ instead of $\delta t \leq C h_\disc$ with $C$ independent of $\ma$), the discrete global entropy estimate \eqref{eq:esch_stab} holds true.

\medskip
Let $\theta_\disc$ be a measure of the regularity of the mesh in the classical finite element sense (see for instance \cite{lat-18-conv}).
Then for all $K\in\mesh$ and $\edge,\edge'\in\edges(K)$, we have $|\edge'|/|D_\edge|\leq C_1(\theta_\disc) \, h_\disc^{-1}$ for some nondecreasing function  $C_1$. Moreover, thanks to classical inverse inequalities, there exists a nondecreasing function $C_2$ such that  $\norm{\bfu}_{\xL^\infty(\Omega)^d}\leq C_2(\theta_\disc) h_\disc^{-\frac d2}\norm{\bfu}_{\xL^2(\Omega)^d}$ for all $\bfu\in\xbfH_{\edges}(\Omega)$.

\medskip
Since the initial data satisfies \eqref{eq:u0_rho0_wp_exp}, the initial total energy is uniformly bounded with respect to $\ma$.
We denote by $C_0$ a uniform upper bound (which depends on the constant $C$ in \eqref{eq:u0_rho0_wp_exp} and on the spatial discretization $\disc$), which thus satisfies for all $\ma>0$:
\[
\frac 1 2  \sum_{\edge\in \edgesint} |D_\edge| \ \rho^{-1}_\Ds\ |\bfu^0_\edge|^2
+ \frac 1 {\ma^2} \sum_{K\in \mesh}|K| \, \Pi_\gamma(\rho^0_K) \leq C_0.
\]
We define:
\[
 C(d,\theta_\disc,C_0) := \frac{1}{2\sqrt{2}d\,C_1(\theta_\disc)\,C_2(\theta_\disc) C_0^{\frac 12}}.
\]
We have the following result.

\begin{proposition}
Let $\eta$ be a fixed small parameter in $(0,1)$. If the time step $\delta t$ satisfies the following CFL condition, which is independent of the Mach number:
\begin{equation} \label{CFL-ind}
\delta t \leq (1-\eta) \min \left ( C(d,\theta_\disc,C_0) h_\disc^{1+\frac d2} \ ,
\  \min\limits_{\edge\in\edgesint}  \, \frac{2\,|D_\edge|}{ \varrho(\mcal{A})} \right ),
\end{equation}
then there exists $\bar \ma$ (depending on $\gamma$, the spatial discretization $\disc$, the constant $C$ in \eqref{eq:u0_rho0_wp_exp} and on $\eta$) such that for all $\ma\in(0,\bar\ma)$:
\begin{equation} \label{eq:global-ent-expl}
\max\limits_{n=0,..,N}\ \Bigl\lbrace \frac 1 2  \sum_{\edge\in \edgesint} |D_\edge| \ \rho^{n-1}_\Ds\ |\bfu^n_\edge|^2
+\frac 1 {\ma^2} \sum_{K\in \mesh}|K| \, \Pi_\gamma(\rho^n_K) \Bigr \rbrace \leq C_0.
\end{equation}
\end{proposition}

\medskip
\begin{remark}
As a by-product of the proof, we can see that, if the time step $\delta t$ satisfies the CFL condition \eqref{CFL-ind}, then it satisfies the classical CFL condition \eqref{CFL} independently of the Mach number $\ma$, and this also allows to prove \eqref{eq:global-ent-expl}.
Actually the proof given hereunder consists in proving \eqref{CFL} and \eqref{eq:global-ent-expl} simultaneously, thanks to an induction process.
\end{remark}

\begin{proof}
Let $\delta t$ satisfy the CFL condition \eqref{CFL-ind}.
Let $(\alpha^n_\ma)_{n=0,..,N}$ and $(\beta^n_\ma)_{n=0,..,N}$ be the two sequences defined by:
\begin{align*}
 &\alpha^n_\ma := \min \left ( 
 \min\limits_{\edge\in\edgesint}  \, \frac{\rho_\Ds^n}{  |D_\edge|^{-1}\sum_{\edged\in\edgesd(D_\edge)}\, \bigl( \fluxd(\rho^n,\bfu^n)\bigr )^{-}}
 \ , \
 \min\limits_{\edge\in\edgesint}  \, \frac{2\, \rho_\Ds^n }{  |D_\edge|^{-1}\varrho(\mcal{A})} \right ),\\[2ex]
 & \beta^n_\ma := \frac 1 2  \sum_{\edge\in \edgesint} |D_\edge| \ \rho^{n-1}_\Ds\ |\bfu^n_\edge|^2+\frac 1 {\ma^2} \sum_{K\in \mesh}|K| \, \Pi_\gamma(\rho^n_K).
\end{align*}
By Lemma \ref{lmm:loc_ener_esch}, we know that for all $0\leq n \leq N-1$ and for all $\ma>0$, $\delta t \leq \alpha^n_\ma \Longrightarrow \beta^{n+1}_\ma \leq \beta^n_\ma$. Let us prove that for all $0\leq n \leq N-1$:
\begin{multline}
\label{Hyp-rec}
\Bigl(\exists \,  \ma^n >0 \ \text{such that} \ \beta^{n+1}_\ma \leq \beta^n_\ma \leq .. \leq \beta^0_\ma \leq C_0  \ \text{for all}  \ \ma\in(0,\ma^n)\Bigr) 
\\
\Longrightarrow \qquad \Bigl ( \exists \,  \ma^{n+1}>0 \ \text{with} \ \ma^{n+1}\leq \ma^n \  \text{such that} \ \delta t \leq \alpha^{n+1}_\ma \ \text{for all}  \ \ma\in(0,\ma^{n+1}) \Bigr). 
\end{multline}

\bigskip
Let us assume that the left-hand side of the above implication holds true for some $n$ such that $0 \leq n \leq N-1$.
Then, for every $\edge=K|L\in\edgesint$, we have the following inequalities:
\begin{equation} \label{temp1}
\begin{aligned}
|D_\edge|^{-1}\sum_{\edged\in\edgesd(D_\edge)}\, \bigl( \fluxd(\rho^{n+1},\bfu^{n+1})\bigr )^{-} 
& \leq 2d\, \max\limits_{\edge'\in\edges(K)\cup\edges(L)} |D_\edge|^{-1} \, |F_{K,\edge'}(\rho^{n+1},\bfu^{n+1})| \\
& \leq 2d\, \max\limits_{\edge'\in\edges(K)\cup\edges(L)} \frac{|\edge'|}{|D_\edge|}\, |\rho^{n+1}_{\edge'}| \, |\bfu^{n+1}_{\edge'}|\\[1ex]
& \leq  2d\, C_1(\theta_\disc) h_\disc^{-1} \bigl (\max\limits_{K\in\mesh}|\rho^{n+1}_K| \bigr ) \bigl (\max\limits_{\edge\in\edgesint} |\bfu^{n+1}_\edge| \bigr).
\end{aligned}
\end{equation}
The first inequality in \eqref{temp1} follows from hypothesis (H3) in \eqref{eq:F_bounded}. 
We then remark that $\max\limits_{\edge\in\edgesint} |\bfu^{n+1}_\edge|=\norm{\bfu^{n+1}}_{\xL^\infty(\Omega)^d}\leq C_2(\theta_\disc) h_\disc^{-\frac d2}\norm{\bfu^{n+1}}_{\xL^2(\Omega)^d}$ with
\[
\begin{aligned}
\norm{\bfu^{n+1}}_{\xL^2(\Omega)^d} 
&\leq \Bigl ( \frac{2}{\min\limits_{K\in\mesh} \rho_K^n}\Bigr)^{\frac 12} \, 
\Bigl(   \frac 1 2  \sum_{\edge\in \edgesint} |D_\edge| \ \rho^n_\Ds\ |\bfu^{n+1}_\edge|^2
+\frac 1 {\ma^2} \sum_{K\in \mesh}|K| \, \Pi_\gamma(\rho^{n+1}_K) \Bigr)^{\frac 12}.
\end{aligned}
\]
Hence, since the left-hand side of \eqref{Hyp-rec} is assumed to hold true, we obtain that
\[
\norm{\bfu^{n+1}}_{\xL^2(\Omega)^d} \leq \Bigl ( \frac{2}{\min\limits_{K\in\mesh} \rho_K^n}\Bigr)^{\frac 12} \, C_0^{\frac 12}.
\]
Injecting in \eqref{temp1}, we obtain:
\[
\min\limits_{\edge\in\edgesint}\,\frac{\rho_\Ds^{n+1}}{|D_\edge|^{-1}\sum_{\edged\in\edgesd(D_\edge)}\bigl( \fluxd(\rho^{n+1},\bfu^{n+1})\bigr )^{-}}
\geq 
C(d,\theta_\disc,C_0) \, h_\disc^{1+\frac d2} \, \frac{\min\limits_{K\in\mesh}\rho_K^{n+1}}{\max\limits_{K\in\mesh}\rho^{n+1}_K} \, \Bigl(\min\limits_{K\in\mesh} \rho_K^n\Bigr)^{\frac 12}.
\]
Invoking once again the left hand side of \eqref{Hyp-rec}, we have $\beta_\ma^{n+1} \leq \beta_\ma^n \leq C_0$ for all $\ma\in(0,\ma^n)$, which implies that 
$\norm{\Pi_\gamma(\rho^{n+1})}_{\xL^1(\Omega)}\leq C_0 \ma^2$ and $\norm{\Pi_\gamma(\rho^n)}_{\xL^1(\Omega)} \leq C_0 \ma^2$ for all $\ma\in(0,\ma^n)$. By the results of Lemma \ref{lem:pi}, this implies that for all $K\in\mesh$, $\rho_K^n\to 1$ and $\rho_K^{n+1}\to 1$ as $\ma\to0$. Hence, there exists $0<\ma^{n+1}\leq\ma^n$ such that for all $\ma\in(0,\ma^{n+1})$:
\begin{equation}
\label{temp2}
\min\limits_{\edge\in\edgesint}\,\frac{\rho_\Ds^{n+1}}{|D_\edge|^{-1}\sum_{\edged\in\edgesd(D_\edge)}\bigl( \fluxd(\rho^{n+1},\bfu^{n+1})\bigr )^{-}}
\geq (1-\eta) C(d,\theta_\disc,C_0) \, h_\disc^{1+\frac d2}.
\end{equation}
Obviously, $\ma^{n+1}$ depends on $\eta$.
It also depends on $\gamma$ and on the spatial discretization (namely $\theta_\disc$ and $h_\disc$) since one has to bound the $\xL^\infty$-norm of $\rho^n-1$ (resp. of $\rho^{n+1}-1$) by the $\xL^{\min(2,\gamma)}$-norm of $\rho^n-1$  (resp. of $\rho^{n+1}-1$) thanks to an inverse inequality, combined with the estimates of Lemma \ref{lem:pi}.
Upon diminishing $\ma^{n+1}$, we can also prove that for all $\ma\in(0,\ma^{n+1})$:
\begin{equation} \label{temp3}
\min_{\edge\in\edgesint}  \, \frac{2\, \rho_\Ds^{n+1} }{  |D_\edge|^{-1}\varrho(\mcal{A})} \geq 
\Big(\min_{\edge\in\edgesint}  \, \frac{2\, |D_\edge|}{\varrho(\mcal{A})} \Big) \, \bigl(\min_{K\in\mesh}\rho_K^{n+1} \bigr) \geq
(1-\eta) \, \min_{\edge\in\edgesint}  \, \frac{2\, |D_\edge|}{\varrho(\mcal{A})}.
\end{equation}
Combining \eqref{temp2} and \eqref{temp3}, we obtain that $\delta t \leq \alpha^{n+1}_\ma$ for all $\ma\in(0,\ma^{n+1})$.
Hence, by a straightforward induction process, the proposition is proved  with $\bar \ma := \ma^{N-1}$ provided that there exists $\ma^0$ such that $\beta^1_\ma\leq\beta_\ma^0$ for all $\ma\in(0,\ma^0)$.
It is sufficient to prove that for some $\ma^0>0$,  $\delta t \leq \alpha_\ma^0$ for all $\ma\in(0,\ma^0)$. 
Following similar steps as above, this is easily proved with $\ma^0$ only depending on the constant $C$ in \eqref{eq:u0_rho0_wp_exp} and on the discretization.
\end{proof}
%
% -------------------------------------------------------------------------------------------------------------------
%
\bibliographystyle{abbrv}
\bibliography{lowM}
\end{document}